\documentclass{amsart}

 \pdfoutput=1

\usepackage{hyperref} 
\usepackage[dvips]{graphicx}
\usepackage[T1]{fontenc}
\usepackage{mathptmx}

\usepackage{color}

\newcommand{\SSS}{\mathbb{S}}
\newcommand{\C}{\mathbb{C}}

\newcommand{\uhat}{\widehat{\mathcal{U}}^\C_+}
\newcommand{\LL}{\mathbb{L}}

\renewcommand{\Re}{\operatorname{Re}}
\renewcommand{\Im}{\operatorname{Im}}
\newcommand{\phihat}{\hat \Phi}
\newcommand{\phiom}{\hat \Phi _\omega}
\newcommand{\xihat}{\hat \xi}

\newcommand{\fomega}{\hat F _\omega}
\newcommand{\xiom}{\hat \xi _\omega}
\newcommand{\fhat}{\hat F}
\newcommand{\bhat}{\hat B}

\newcommand{\bbar}{\begin{pmatrix}}
\newcommand{\ebar}{\end{pmatrix}}

\newcommand{\bdm}{\begin{displaymath}}
\newcommand{\edm}{\end{displaymath}}
\newcommand{\beq}{\begin{equation}}
\newcommand{\beqa}{\begin{eqnarray}}
\newcommand{\beqas}{\begin{eqnarray*}}
\newcommand{\eeq}{\end{equation}}
\newcommand{\eeqa}{\end{eqnarray}}
\newcommand{\eeqas}{\end{eqnarray*}}
\newcommand{\dd}{\textup{d}}
\newcommand{\EEE}{{\mathbb E}}

\newcommand{\bigcell}{\mathcal{B}_{1,1}}

\newcommand{\sym}{\mathcal{S}}
\newcommand{\Ad}{\textup{Ad}}

\newcommand{\enormal}{{\bf n}_E}
 \newcommand{\real}{{\mathbb R}}

\newcommand{\stilde}{\widetilde{\mathcal{S}}}
\newcommand{\uu}{\mathcal{U}} 

   \newtheorem{theorem}{Theorem}[section]
   \newtheorem{proposition}[theorem]{Proposition}
   
   \newtheorem{lemma}[theorem]{Lemma}
   \newtheorem{definition}[theorem]{Definition}

 \theoremstyle{remark}

\numberwithin{equation}{section}

\begin{document}

\title[Singularities of  CMC surfaces in $L^3$]{Singularities of Spacelike Constant Mean Curvature Surfaces in Lorentz-Minkowski Space}

\begin{abstract}
We study singularities of spacelike,  constant (non-zero) mean curvature (CMC) surfaces in the Lorentz-Minkowski $3$-space $L^3$.  We show how to solve the singular Bj\"orling
problem for such surfaces, which is stated as follows: given a real analytic null-curve
$f_0(x)$,  and a real analytic null vector field $v(x)$ parallel to the tangent
field of $f_0$, find a conformally parameterized (generalized) CMC $H$ surface in $L^3$ which contains this  curve as a singular set and such that the partial
derivatives $f_x$ and $f_y$ are given by $\frac{\dd f_0}{\dd x}$ and $v$ along the curve.
Within the class of generalized surfaces considered, the solution is unique
and we give a formula for the generalized Weierstrass data for this surface.
This gives a framework for studying the singularities of non-maximal CMC surfaces
in $L^3$. We use this to find the Bj\"orling data -- and holomorphic potentials -- which characterize
cuspidal edge, swallowtail and cuspidal cross cap singularities.
\end{abstract}

\author{David Brander}
\address{Department of Mathematics\\ Matematiktorvet, Building 303 S\\
Technical University of Denmark\\
DK-2800 Kgs. Lyngby\\ Denmark}
\email{D.Brander@mat.dtu.dk}

\keywords{Differential geometry, integrable systems, Bj\"orling problem, prescribed mean curvature}

\subjclass[2000]{Primary 53A10; Secondary 53C42, 53A35}


\maketitle

\section{Introduction}
 Spacelike constant
mean curvature (CMC) surfaces in $(2+1)$-dimensional space-time $\LL^3$
were studied in \cite{brs} and \cite{inoguchi} using a generalized Weierstrass representation whereby the
surface is represented by a holomorphic map into a loop group. This is an application of the method of Dorfmeister, Pedit and Wu (DPW) \cite{DorPW} for harmonic maps into symmetric
spaces. In the non-compact case, the Iwasawa decomposition of the loop group, used to
construct the solutions, is only valid on an open dense set, the \emph{big cell}.
  It was shown in \cite{brs}
that singularities of the CMC surface 
arise as the boundary of the big cell is encountered. Here we will analyze these singularities and show how to construct CMC surfaces with
prescribed singular curves, and prescribed types of singularities, via a singular Bj\"orling formulation.

One of the obstructions to the effective use of
integrable systems methods for solving global problems in geometry has been
 the break-down of the loop group decompositions used to construct solutions.  
A motivating factor here is to understand and make use of the big cell boundary
behaviour.

\subsection{Singularities of maximal surfaces and fronts}
In the context of surfaces in Euclidean 3-space $\EEE^3$, a \emph{frontal} is a 
differentiable map $f: M^2 \to \EEE^3$, from a surface $M$,
which has a well defined normal direction, that is, a map
$\enormal: M^2 \to \SSS^2 \subset \EEE^3$ which is orthogonal to $f_*(TM^2)$.  If the map $(f,\enormal)$
is an immersion, then $f$ is called a \emph{(wave) front}.  
A singular point of any smooth map $f: M^2 \to \EEE^3$ is one where $f$ is not immersed,
and singular points $p_1$ and $p_2$ of $f_1: M_1^2 \to \EEE^3$ and 
$f_2: M_2^2 \to \EEE^3$ are called diffeomorphically 
equivalent if there exist local diffeomorphisms of the corresponding spaces
which commute with these maps.  
A theory of the singularities of fronts can be found in Arnold \cite{arnold1990}.
 Geometric concepts, such as curvature and completeness, for surfaces with singularities have been defined by Saji, Umehara and Yamada in \cite{suy}.
 
In this article we will encounter three standard singularities: 
the \emph{cuspidal edge},
given by $f(u,v) = (u^2, \, u^3, \, v)$, the \emph{swallowtail} given by 
$(3u^4 + u^2v, \, 4u^3+2uv, \, v)$ and the \emph{cuspidal cross cap} given by $(u, \, v^2, \, uv^3)$ (Figure \ref{figure1}).  The first two singularities are
fronts, but the third is only a frontal.

\begin{figure}[here]  
\begin{center}
\includegraphics[height=27mm]{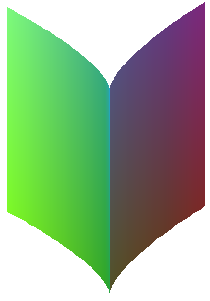} \hspace{10mm}
\includegraphics[height=27mm]{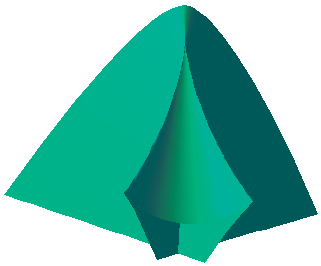} \hspace{10mm}
\includegraphics[height=27mm]{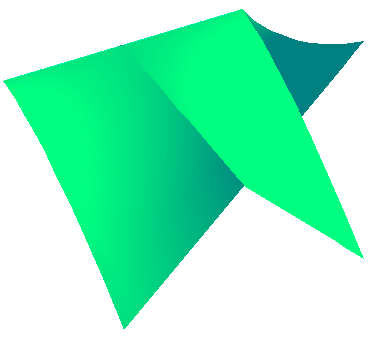}
\end{center}
\caption{Left to right: Cuspidal edge, swallowtail and cuspidal cross cap.}
\label{figure1}
\end{figure}
 
  A  point to note is that if one wants a sensible theory of singularities, for example if one
 would like to classify singularities for a specific type of surface, then one needs to 
 consider \emph{generic} singularities, that is singularities which persist under 
 continuous deformations of the surface through the appropriate class.  
 If one considers the class of $\mathcal{C}^\infty$ maps of 2-manifolds into 3-manifolds,
 Whitney showed that generic singularities are cross caps \cite{whitney1944}.

Fronts and frontals arise naturally within the context of integrable systems -- 
very often it is exactly such  surfaces, rather than immersions, which
are produced via loop group constructions.  Conversely, for many geometric problems,
it is more or less unavoidable to consider surfaces with singularities: for example
it is well known that there is no complete immersion of the hyperbolic plane into
$\EEE^3$, and for the case of spacelike maximal (mean curvature zero) surfaces in $\LL^3$ 
the only complete immersion is the plane.  For these two examples, generic singularities
have been classified: for constant Gauss curvature surfaces in $\EEE^3$, Ishikawa
and Machida \cite{ishimach} showed that they consist of cuspidal edges and
swallowtails;  for maximal surfaces in $\LL^3$, Fujimori, Saji, Umehara and Yamada \cite{umyam2006, fsuy} showed that the generic singularities are all three of those shown
in Figure \ref{figure1}.

Recently there have been a number of interesting studies of maximal surfaces and their singularities: the reader is referred to articles such as \cite{aliaschavesmira2003, fls2005, fl2007, fls2007, kimyang2007, fsuy} and the references therein.
Most closely related to the present article are the classification of generic
singularities \cite{umyam2006, fsuy} already mentioned, and the work of
Y.W. Kim and S.D. Yang \cite{kimyang2007} on the singular Bj\"orling problem for maximal surfaces.

\subsection{The Bj\"orling problem}
The classical Bj\"orling problem for minimal surfaces in $\EEE^3$ is to find the unique
minimal surface containing a given real analytic curve with prescribed tangent
planes along the curve (see \cite{dhkw}). The solution is obtained from the 
initial data by an analytic extension and an elementary formula in terms
of integrals. Since the solution is tied to the Weierstrass representation
of minimal surfaces in terms of holomorphic data, one has a similar
construction for regular maximal surfaces in $\LL^3$, given in \cite{aliaschavesmira2003}, which also have such a 
holomorphic representation. More generally, Kim and Yang \cite{kimyang2007} show that there
is also a solution when the initial curve is null (which implies that the
surface is not immersed there). Instead of prescribing the tangent plane 
along the curve, one seeks a surface which is conformally immersed except 
along the curve, with coordinates $z=x+iy$, and where the curve is given by $\{ y=0 \}$, and then prescribes the value of $f_y$, a null vector field
parallel to $f_x$. Note that null vectors are orthogonal if and only if they are
parallel, so this makes sense in terms of the conformal coordinates.
   One can then use this construction to study the
singularities of maximal surfaces.

As a generalization of the Weierstrass representations for minimal and
maximal surfaces, one has the DPW method for CMC $H \neq 0$ surfaces in
both $\EEE^3$ and $\LL^3$.  In \cite{bjorling}, it was shown that one
could use this method to solve the generalization of the Bj\"orling 
problem to non-minimal CMC surfaces in $\EEE^3$.  It is clear that 
essentially the same
construction works for \emph{regular} CMC $H \neq 0$ surfaces in $\LL^3$,
and we will show below that the \emph{singular} Bj\"orling problem can
also be solved for non-maximal CMC surfaces.  The main obstacle which
needs to be circumvented is that the DPW method depends on the
use of an $SU_{1,1}$ frame (extended to the loop group) and then a loop
group decomposition to go to the holomorphic data.  This $SU_{1,1}$ frame
is not defined along the singular curve, because the (Lorentzian) unit
 normal becomes lightlike and blows up.    Below, we will get around this by
 defining a special $SU_{1,1}$ ``frame", called the \emph{singular frame},
  along the curve, the definition of
 which is motivated by our analysis of the loop group construction.

\subsection{The DPW method}

The generalized Weierstrass representation for spacelike CMC surface in $\LL^3$ follows 
the same logic as that for CMC surface in Euclidean 3-space: in the maximal case, where
the mean curvature $H$ is zero, there is a Weierstrass representation in terms of a pair
of holomorphic functions, just as for minimal surfaces, related to the fact that the
Gauss map is holomorphic.  For the non-maximal case, the Gauss map is harmonic but not
holomorphic, and one can instead use the holomorphic representation for harmonic maps
given in \cite{DorPW}.  The only real difference from the Euclidean case is the non-compactness of the isometry
group, leading to an incomplete picture of what is actually constructed from the given
holomorphic data.  For more details and references, see \cite{brs}.

The DPW construction described in \cite{brs} is as follows:
A CMC $H$ immersion $f: \Sigma \to \LL^3$ from a Riemann surface into Minkowski 3-space
can be represented by a certain type of holomorphic map 
$\phihat: \Sigma \to \Lambda SL(2,\C)_\sigma$  into the twisted loop group of 
smooth maps from the unit circle into $SL(2,\C)$.  The map $\phihat$ is called a \emph{holomorphic
extended frame} for $f$.  
In connection with the Iwasawa decomposition with respect to the non-compact
real form
$\Lambda SU_{1,1}$, the loop group $\Lambda SL(2,\C)_\sigma$ can be written as a disjoint union
$\mathcal{B}_{1,1} \cup \mathcal{P}_1 \cup \mathcal{P}_2 \cup \mathcal{P}_3 \cup ...$. 
The set $\mathcal{B}_{1,1}$ is open and dense in $\Lambda SL(2,\C)_\sigma$, and
is called the (Iwasawa) big cell. 
As a converse to the above statement concerning $f$, given a holomorphic extended frame, 
if we restrict to $\Sigma^\circ := \phihat^{-1}(\mathcal{B}_{1,1})$, one obtains a CMC $H$ immersion into $\LL^3$.
 Behaviour of the surface as the largest two small cells, $\mathcal{P}_1$ and $\mathcal{P}_2$,
 are approached  was examined in \cite{brs}, and it was shown that the CMC surface 
 extends continuously to $\phihat^{-1}(\mathcal{P}_1)$, but is not immersed there, 
 and that the surface blows up as $\phihat^{-1}(\mathcal{P}_2)$ is approached.

 \subsection{Results of this article}
  As we are interested
 in finite singularities, we define  a \emph{generalized CMC $H$ surface} to be a
 map $f$ obtained from a holomorphic extended frame $\phihat$, restricted to 
 $\Sigma_s := \phihat^{-1}(\mathcal{B}_{1,1} \cup \mathcal{P}_1)$. This includes all
 regular CMC $H$ surfaces, as one can always find a holomorphic extended frame for 
a regular surface which takes values in the big cell $\bigcell$.  We know that 
 the singular set $C := \phihat^{-1}(\mathcal{P}_1)$, where $f$ is not immersed,
  is locally given
 as the zero set of a non-constant real analytic function.
 We say that $z_0 \in C$ is \emph{weakly non-degenerate} if $\phihat$ maps some
 open curve containing $z_0$ into $\mathcal{P}_1$. This is simply the weakest 
condition needed to consider the singular Bj\"orling construction, and holds
for a generic point in $C$.

The main results of this article can be summarized as Theorem \ref{bjorling},
Theorem \ref{swallowtailthm} and Theorem \ref{crosscapsthm}.
The first of these results is the solution of the singular Bj\"orling problem
for CMC surfaces in $\LL^3$. It essentially 
says that given a real analytic curve $f_0: J \to \LL^3$, 
from some interval $J \subset \real \subset \C$, such that $\frac{\dd f_0}{\dd x}$ is a null
vector field, and given a real analytic vector field $v: J \to \LL^3$ which is 
proportional to $\frac{\dd f_0}{\dd x}$, then, for any constant $H>0$, there is a unique,
weakly non-degenerate, generalized
CMC $H$ surface $f$  satisfying 
$f \big|_J = f_0$ and $\frac{\partial f}{\partial y} \big|_J = v$.  It also gives a formula for the holomorphic potential
for the surface in terms of analytic extensions of the data specified along $J$.

The other two results mentioned, Theorems \ref{swallowtailthm} and
\ref{crosscapsthm}, give the conditions on the Bj\"orling data for the singularity
at a point $z_0 \in J$ to be diffeomorphic to a cuspidal edge, swallowtail or 
cuspidal cross cap. The conditions are simple: for the given Bj\"orling
data, one can always write $\frac{\dd f_0}{\dd x} = s \,[ \,\cos \theta, 
\sin \theta, 1 \,]$ and $v(x) = t \, [ \, \cos \theta, \sin \theta, 1 \,]$,
where $s$, $t$, and $\theta$ are $\real$-valued, and we assume that
$s$ and $t$ do not vanish simultaneously to avoid branch points.
 Then $s(0) \neq 0 \neq t(0)$
corresponds to a cuspidal edge at the coordinate origin; $s(0) = 0$ and $s^\prime(0) \neq 0$
corresponds to a swallowtail;  $t(0) = 0$ and $t^\prime(0) \neq 0$ is a
cuspidal cross cap (see Figure \ref{figure2}).

\begin{figure}[here]  
\begin{center}
\includegraphics[height=45mm]{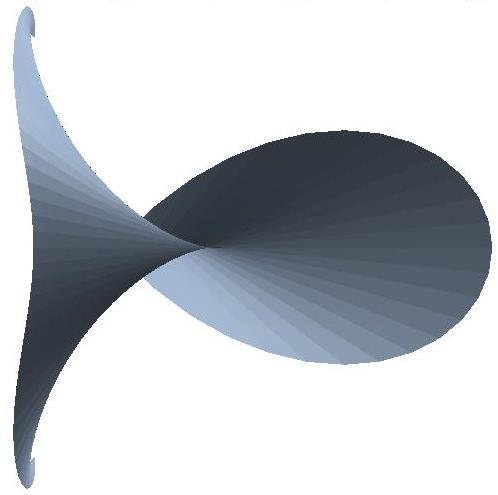} \hspace{10mm}
\includegraphics[height=45mm]{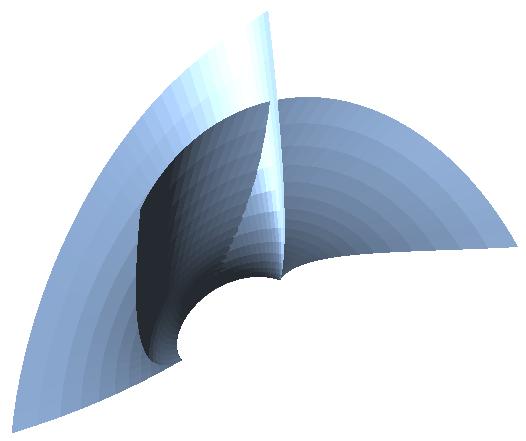}
\end{center}
\caption{Left: a CMC swallowtail singularity, computed numerically from the Bj\"orling data $s(x) = x$, $t(x)=1$, $\theta(x)=0.0001 x$. Right: a CMC cuspidal cross cap, computed from the data $s(x) = 1-x$, $t(x)=x$, $\theta(x)=0.001 x$. The images have been rescaled in the  direction $e_2+e_3$.}
\label{figure2}
\end{figure}
\subsection{Open questions}  \label{openquestions}
It appears plausible that the three types of singularities 
just mentioned are the  generic singularities for CMC surfaces in $\LL^3$,
just as was shown for maximal surfaces in \cite{fsuy}. To prove this using the constructions here, one would first need to show
that generic singularities do not occur on higher small cells $\mathcal{P}_j$, for
$j>2$.  This seems likely, because the codimensions 
 of the small cells $\mathcal{P}_j$ in the loop group increase (pairwise) as $j$ 
 increases. Regardless of genericity, knowledge of the behaviour of the surface close
  to such points would also be interesting to have.
 
 \subsection{Alternative approaches: the Kenmotsu formula representation}
 An alternative to the DPW method is the Kenmotsu formula \cite{kenmotsu} for CMC surfaces in $\EEE^3$, adapted to spacelike CMC surfaces in $\LL^3$
by Akutagawa and Nishikawa  in \cite{akutagawanishikawa}.  This is also a
generalization of the Weierstrass representation for minimal/maximal surfaces,
as a formula in terms of the harmonic Gauss map.  In contrast to the DPW method,
one is still left with the problem of constructing the harmonic map.  The
Kenmotsu-Akutagawa-Nishikawa approach has been used by Y. Umeda \cite{umeda}
to study CMC surfaces with singularities in $\LL^3$, giving the conditions on
the harmonic Gauss map corresponding to cuspidal edges, swallowtails and cuspidal cross caps,
as well as some examples.  It is stated as an open problem whether or not a CMC
cuspidal cross cap exists: here we give a positive answer to this question, and,
in principal, construct all such singularities from their Bj\"orling data.

\section{Background material} \label{dpwsummarysection}
This section is a short summary of results in \cite{brs}. We use mostly  the same
notation and definitions here.  \textbf{Notational convention:} If $\hat X$
is some object with values in the loop group, with loop parameter $\lambda$, then
dropping the hat means the object is evaluated at $\lambda =1$:
\bdm
X := \hat X \Big|_{\lambda =1}.
\edm

\subsection{The loop group formulation for CMC surfaces in $\LL^3$} 
We use the basis 
\bdm
 e_1 := \bbar 0 & 1 \\ 1 & 0 \ebar, \hspace{.5cm} e_2 :=
 \bbar 0 & i \\ -i & 0 \ebar ,  \hspace{.5cm}
 e_3:= \bbar i & 0 \\ 0 & -i \ebar,
\edm
for the Lie algebra $\mathfrak{su}_{1,1}$.
With respect to the Killing metric, $\langle X,Y \rangle = \tfrac{1}{2} \text{trace} (XY)$, these vectors are orthogonal and normalized as
follows:
\bdm
\langle e_1,e_1 \rangle = \langle e_2,e_2 \rangle = - \langle e_3,  e_3 \rangle = 1,
\edm
 so we identify $\mathfrak{su}_{1,1}$ with 
the Lorentz-Minkowski space $\LL^3 = \real^{2,1}$, and also
use the notation $[a,b,c]^T= a e_1 + b e_2 + c e_3$ for a point in $\LL^3$.

Let $G$ be the subgroup of $SL(2,\C)$ consisting of elements of either
$SU_{1,1}$ or of $ie_1 \cdot SU_{1,1}$,
\beq
G =  \left\{ \bbar a & b \\ \varepsilon \bar b & \varepsilon \bar a \ebar ~ \Big|~ a,~b \in \C, \,\,\,  \varepsilon (a \bar a -  b  \bar b) = 1,
 \,\,\, \varepsilon = \pm 1 \right\}. \label{gmatrix}
\eeq
The Lie algebra of $G$ is $\mathfrak{g} = \mathfrak{su}_{1,1}$.

The twisted loop group $\uu := \Lambda G_\sigma$ consists of maps, $x: \SSS^1 \to G$, from the 
unit circle into $G$, such the diagonal and off-diagonal elements of the
matrix are even and odd functions of the $\SSS^1$ parameter $\lambda$.
All loops are of a suitable smoothness class so that the loop groups are Banach Lie groups.
An element of $\uu$ can again be written as in (\ref{gmatrix}), where
now $a$ and $b$ are respectively even and odd functions of $\lambda$. 
We will generally be considering loops which
extend holomorphically to an annulus around $\SSS^1$, and for these the holomorphic extensions
of
$\bar a$ and $\bar b$   respectively  have Fourier expansions $a^*(\lambda):= \overline{(a(1/\bar{\lambda}))}$ and
$b^*(\lambda):= \overline{(b(1/\bar{\lambda}))}$.
We can write 
\bdm
\uu := \Lambda G_\sigma = \uu_{1} \cup \uu_{-1},
\edm
where the $\varepsilon$ in  $\uu_{\varepsilon}$ corresponds to that in
(\ref{gmatrix}). We also have $\uu_{1} = \Lambda SU_{1,1}$ and
$\uu_{-1} = \tiny{\bbar 0 & \lambda i \\ \lambda^{-1} i & 0 \ebar} \cdot \uu_{1}$.
The Lie algebra, $Lie(\uu) = Lie(\uu_1)$, of $\uu$, consists of loops of matrices with analogous properties
to those in $\uu$, replacing the determinant $1$ condition with the trace zero condition.

The complexification of $\uu$ is $\uu^\C := \Lambda SL(2,\C)_\sigma$,
the group of loops in $SL(2,\C)$ which again have the twisted condition 
on diagonal/off-diagonal elements mentioned above.
Let ${\mathbb D}_{\pm}:= \{ \lambda \in \C \cup \{ \infty \} ~ \big| ~ |\lambda|^{\pm 1} < 1 \}$.
Three subgroups of $\uu^\C$ that we also use are: 
\beqas
\begin{aligned}
& \uu^\C_\pm :=   \{ \bhat \in \uu^\C ~ \big| ~ 
       \bhat \text{ extends holomorphically to } {\mathbb D}_\pm \}, \\
& \uhat := \{ \bhat \in \uu^\C_+ ~\big|~ \bhat \big|_{\lambda=0} = \tiny{\bbar \rho & 0 \\ 0 & \rho^{-1} \ebar}, ~\rho \in \real, ~\rho>0 \}. 
\end{aligned}
\eeqas


Let $\Sigma$ be a simply connected non-compact
Riemann surface, and suppose $f: \Sigma \to \LL^3$ is a conformal 
spacelike immersion with constant mean
curvature $H \neq 0$, or an $H$-surface.  Without loss of generality,
we assume that $H>0$, the sign being a matter of orientation.
 If $z= x+iy$ is a local coordinate,
there is a function
$u: \Sigma \to \real$ such that the metric is given  by
$\dd s^2 = 4e^{2u}(\dd x^2 + \dd y^2)$.
The \emph{coordinate frame}  $F: \Sigma \to SU_{1,1}$ is well defined up to premultiplication by
$\pm I$, by
\beq \label{framedefn}
F e_1 F^{-1} = \frac{f_x}{|f_x|}, \hspace{1cm} 
F e_2 F^{-1} = \frac{f_y}{|f_y|}.
\eeq
Choose the conformal coordinates $x$ and $y$  such that the oriented unit normal
is then given by
$N = F e_3 F^{-1}$.  
The Hopf differential is defined to be $Q \dd z^2$, where
$Q:= \langle N,f_{zz} \rangle = -\langle N_z,f_z \rangle.$ 
The Maurer-Cartan form, $\alpha$, for the frame $F$ is defined to be
$\alpha := F^{-1} \dd F = U \dd z + V \dd \bar{z}$, where 
the connection coefficients $U := F^{-1}F_z$ and 
$V := F^{-1}F_{\bar{z}}$ are given by 
\begin{equation}\label{UhatandVhat}
U = \frac{1}{2} \begin{pmatrix} u_z & -2 i H e^u \\ 
                                   i e^{-u} Q & -u_z \end{pmatrix} 
, \hspace{1cm} V = \frac{1}{2} \begin{pmatrix} -u_{\bar z} & -i e^{-u} \bar Q \\ 
                            2 i H e^u & u_{\bar z} \end{pmatrix} . \end{equation}
The compatibility condition $\dd \alpha + \alpha \wedge \alpha = 0$
is equivalent to the pair of equations
\beq
u_{z \bar z} - H^2 e^{2u}+\tfrac{1}{4}  |Q|^2 e^{-2u} = 0,
\hspace{1cm}
 Q_{\bar z} = 2 e^{2 u} H_z.  \label{compatibility1} 
\eeq
The above structure for $U$ and $V$ are verified by a computation, using
 $H = \frac{1}{8}e^{-2u}\langle f_{xx} + f_{yy}, N \rangle$, 
and 
 \begin{equation}
f_z = 2 e^u F \cdot \bbar  0 & 1 \\ 0 & 0 \ebar
\cdot F^{-1} \; , \hspace{1cm}
f_{\bar z} = 2 e^u F \cdot \bbar  0 & 0 \\ 1 & 0 \ebar
\cdot F^{-1} \; . \end{equation}

We can insert  an $\SSS^1$ parameter $\lambda$ into the $1$-form $\alpha$, defining 
a family
$\hat \alpha := \hat U \dd z + \hat V \dd \bar{z}$, where
\begin{equation}\label{withlambda}
\hat U = \frac{1}{2} \begin{pmatrix} u_z & -2 i H e^u \lambda^{-1} \\ 
                         i e^{-u} Q \lambda^{-1} & -u_z \end{pmatrix} 
, \hspace{.5cm} \hat V = 
\frac{1}{2} \begin{pmatrix} -u_{\bar z} & -i e^{-u} \bar Q \lambda \\ 
                     2 i H e^u \lambda & u_{\bar z} \end{pmatrix} 
\; . \end{equation}
Then the assumption that $H$ is constant is equivalent to the integrability of
$\hat \alpha$ for all $\lambda$. 
 Hence it can be integrated to 
obtain a map $\hat F: \Sigma \to \uu_1$.
Supposing that our coordinate frame $F$ defined above satisfies $F(z_0) = F_0$, at some
point $z_0$, we integrate $\hat \alpha$ with the same initial condition, and
call the map $\hat F: \Sigma \to \uu_1$ thus obtained 
an \emph{extended frame}
for the $H$-surface $f$.

The Sym-Bobenko formula is the map $\sym: \uu \to Lie(\uu)$ given by:
\beq \label{symformula}
\sym(\hat F) :=   -\frac{1}{2H} \left( \hat F e_3 \hat F^{-1} + 2 i \lambda 
\frac{\partial\hat F}{\partial \lambda} \, \hat F^{-1} \right). 
\eeq
We write $\sym_\lambda: \mathcal{U} \to \LL^3$ for the map given by evaluating
this at $\lambda \in \SSS^1$.
If  $\hat F: \Sigma \to \uu_1$ is an extended frame for an $H$-surface $f$, then, up a translation
in $\LL^3$, the surface is retrieved by applying the Sym-Bobenko formula at $\lambda =1$:
\bdm
f = \sym_1 (\hat F)  + \textup{translation}.
\edm
This is verified by computing $\sym_1 (\hat F) _z$ and $\sym_1 (\hat F)_{\bar z}$, using the matrices $\hat U$ and
$\hat V$.  The same computation shows that $\sym_{\lambda_0} (\hat F)$ is also an 
$H$-surface for any
$\lambda_0 \in \SSS^1$.  
For such computations, note that if 
\bdm
\hat G^{-1} \hat G_z = \bbar u_0 & \alpha \lambda^{-1} \\ \beta \lambda^{-1} & -u_0 \ebar, \hspace{1cm}
\hat G^{-1} \hat G_{\bar z} = \bbar  -\bar u_0 & \bar \beta \lambda \\ \bar \alpha \lambda & \bar u_0 \ebar,
\edm
and we set $f^\lambda =\sym_{\lambda}(\hat G)$, then one computes the following formulae:
\beq\label{derivatives}
\hat G^{-1} f^\lambda_z \, \hat G = \frac{2i}{H} \bbar 0 & \alpha \lambda^{-1} \\ 0 & 0 \ebar,
  \hspace{1cm}
\hat G^{-1} f^\lambda_{\bar z} \, \hat G = \frac{2i}{H} \bbar 0 & 0 \\ - \bar \alpha \lambda & 0 \ebar.
\eeq

One can also define a CMC surface with 
extended coordinate frame $\widetilde F$ in the other half of the loop group, $\uu_{-1}$,
by integrating the 1-form $\hat U \dd z + \hat V \dd \bar z$
with the initial condition 
\bdm
\widetilde F (z_0) = W = \bbar 0 & i \lambda \\ i \lambda^{-1} & 0 \ebar.
\edm
 Since $\sym(W \fhat) = \Ad_W \sym (\fhat) + \textup{translation}$ -- where $\Ad_X$ denotes conjugation by $X$ -- and $\Ad_W$ is an isometry of $\LL^3$, this is also a 
CMC surface.  If $\fhat$ is the frame obtained with the
initial condition $\fhat(z_0) = I$, then the relation between the surfaces obtained
at $\lambda =1$ is $\sym _1 (\widetilde F) = \Ad_W \big|_{\lambda=1} \sym_1(\fhat) + \textup{translation}$. The 
coordinate frame
for $\tilde f = \sym_1(\widetilde F)$ satisfies
$\widetilde F e_1
 \widetilde F \big|_{\lambda=1} = \frac{\tilde f_x}{|\tilde f_x|}$ and 
$\widetilde F e_2 \widetilde F \big|_{\lambda=1} = \frac{\tilde f_y}{|\tilde f_y|}$.

More generally, one can show (see, for example, the analogous argument in \cite{bjorling}):
\begin{lemma}  \label{lemma1}
 If $\hat F: \Sigma \to \uu = \uu_1 \cup \uu_{-1}$ is a real analytic 
map the Maurer-Cartan form of which has the form
\beq \label{requiredmcform}
\hat F^{-1} \dd \hat F = \alpha_{-1} \,  \dd z \,  \lambda^{-1} + \hat \beta \dd z + \hat \gamma \dd \bar z,
\eeq
where the loop-algebra valued functions $\hat \beta$ and $\hat \gamma$ extend holomorphically in $\lambda$
to the unit disc, and
with the regularity condition $[\alpha_{-1}]_{12} \neq 0$, then the map ${f}^{\lambda_0} =  
\sym_{\lambda_0}(\hat F)$ is an $H$-surface in $\LL^3$, and the coordinate
frame for this surface is given by $F =  \hat F|_{\lambda_0} \, D$, where $D: \Sigma \to G$ 
is a diagonal matrix-valued function.
\end{lemma}

Note that the Sym-Bobenko formula is invariant under gauge transformations $\hat F \mapsto \hat F D$, where
$D$ is constant in $\lambda$ and diagonal.  It also follows from the fact that 
the 1-form $\hat F^{-1} \dd \hat F$ of Lemma \ref{lemma1} takes values in $Lie(\uu)$ that, in fact,
\bdm
\hat F^{-1} \dd \hat F = \alpha_{-1} \,  \dd z \,  \lambda^{-1} +  
\alpha_0 \, \dd z + \tau(\alpha_0) \, \dd \bar z + 
 \tau(\alpha_{-1}) \,  \dd \bar z \, \lambda,
\edm
where the involution $\tau$ that defines $\mathfrak{g} = \mathfrak{su}_{1,1}$ as a real form of $\mathfrak{sl}(2,\C)$ is given by:
\bdm
\tau (X) := - \Ad_\sigma \bar X ^t, \hspace{1cm} \sigma = \bbar 1 & 0 \\ 0 & -1 \ebar.
\edm

\subsection{Construction of solutions via the DPW method}
By Lemma \ref{lemma1}, the problem of constructing a conformal spacelike
 CMC immersion $f: \Sigma \to \LL^3$
is evidently equivalent to the problem of constructing a real analytic
 map $\hat F: \Sigma \to \uu$, such that $\hat F ^{-1} \dd \hat F$ is of the type given 
 by (\ref{requiredmcform}).  The DPW construction does exactly that, beginning with an
 arbitrary \emph{holomorphic} map $\phihat: \Sigma \to \uu^\C$ which satisfies
 $\phihat^{-1} \dd \phihat = (\beta_{-1} \lambda^{-1} + \beta_0 + ...) \dd z$.

In order to explain this, we first need to state the Iwasawa decomposition of $\uu^\C$.
Define, for a positive integer $m \in {\mathbb Z}^+$,
\bdm
 \omega_{m} = \begin{pmatrix}
1 & 0 \\ \lambda^{-m} & 1
\end{pmatrix}, \;\; \text{$m$ odd}  ; \hspace{1cm}
\omega_{m} = \begin{pmatrix}
1 & \lambda^{1-m} \\ 0 & 1 
\end{pmatrix}, \;\; \text{$m$ even.} 
\edm

\begin{theorem} \label{iwasawathm} ($SU_{1,1}$ Iwasawa 
decomposition \cite{brs})
\begin{enumerate}
\item\label{mainthm-(1)}
The group $\uu^{\C}$ is a disjoint union
\beq \label{globaldecomp}
 \uu^{\C} =  
\bigcell  \sqcup \bigsqcup_{m \in {\mathbb Z}^+ } \mathcal{P}_m,
 \eeq
where
\bdm
\bigcell := \uu \cdot \uu^\C_+,
\edm
is called the \emph{big cell}, and
the \emph{$n$-th small cell} is:
\beq \label{pndecomp}
\mathcal{P}_n := \uu_1 \cdot \omega_n \cdot \uu^\C_+.
\eeq

\item \label{mainthm-(2)}
In the factorization
\beq \label{bciwasawa}
\phihat = \hat F \hat B,  \hspace{1cm} \hat F \in \uu, \hspace{.5cm} \hat B \in \uu^\C_+,
\eeq
of a loop $\phihat \in \bigcell$, 
the factor $\hat F$ is unique up to right
multiplication by an element of the subgroup $\uu^0$ of constant loops in $\uu$.
 Both factors are unique
if we require that $\hat B \in \uhat$,  and with this normalization the product map
$\uu \times \uhat \to \bigcell$ is a 
real analytic diffeomorphism.

\noindent
 \item\label{mainthm-(3)}
The Iwasawa big cell, $\bigcell$, is an open dense subset of
 $\uu^{\C}$. The complement of $\bigcell$ in $\uu^{\C}$ is locally
 given as the zero set of a non-constant real analytic function $\uu^{\C} \to \C$.
\end{enumerate}
\end{theorem}
It is clear from Theorem \ref{iwasawathm} that the big cell $\bigcell$ is naturally
divided into two disjoint open sets corresponding to whether the element $\hat F$ is a
loop in $SU_{1,1}$ or in $i e_1 SU_{1,1}$. We denote these subsets by $\bigcell^+$
and $\bigcell^-$ respectively. 

Now it is easy to check that if  $\phihat: \Sigma \to \bigcell \subset \uu^\C$ satisfies
 $\phihat^{-1} \dd \phihat = (\beta_{-1} \lambda^{-1} + \beta_0 + ...) \dd z$, and
 $\phihat = \hat F \hat B$ is an Iwasawa factorization of $\phihat$, with $\hat F \in \uu$,
 then $\hat F ^{-1} \dd \hat F$ is of the required form (\ref{requiredmcform}).
That is the essential point behind the 
generalized Weierstrass representation for
 $H$-surfaces which will be stated in the next theorem.
\begin{definition}
 A \emph{standard (holomorphic) potential} on a Riemann surface $\Sigma$
is a holomorphic 1-form $\xihat \in Lie(\uu^\C) \otimes \Omega^{1,0} (\Sigma)$, the Fourier expansion of which begins at $\lambda^{-1}$:
\bdm
\xihat = \sum_{i=-1}^\infty \beta_i \lambda^i \dd z,
 \hspace{1cm} \beta_i: \Sigma \to \mathfrak{sl}(2,\C), \textup{ holomorphic},
\edm
and with the regularity condition on the (1,2) component of $\beta_{-1}$:
\bdm
[\beta_{-1}]_{12}\,(z) \neq 0, \hspace{.5cm} \forall z \in \Sigma.
\edm
 \end{definition}
\begin{theorem} \label{dpwthm} \cite{brs}. \hspace{.2cm}
Let $\xihat$ be a standard holomorphic potential 
on a simply-connected Riemann surface  $\Sigma$.
 Let $\phihat :\Sigma \to \uu^\C$ be 
a solution of 
\bdm
\phihat^{-1} \dd \phihat= \xihat.
\edm  
Define the open set 
$\Sigma^\circ := \phihat ^{-1} (\bigcell)$.
Assume that the map $\phihat$,  maps at least one point into $\bigcell$,
so that $\Sigma^\circ$ is not empty,
 and take any  $G$-Iwasawa splitting pointwise on $\Sigma^\circ$: 
\beq \label{thmsplit}
\phihat = \hat F \hat B, \hspace{1.5cm} \hat F \in \uu, \hspace{.2cm} \hat B \in \uu^\C_+.
\eeq
  Then for any $\lambda_0 \in \SSS^1$, the map
$f^{\lambda_0} := \sym_{\lambda_0}(\hat F) \,: \, \Sigma^\circ \to \LL^3$, given by
the Sym-Bobenko formula \eqref{symformula}, is a conformal spacelike
CMC $H$ immersion,
and is independent of the choice of $\hat F \in \uu$ in (\ref{thmsplit}).

Conversely, let $\Sigma$ be a noncompact Riemann surface.  Then 
any non-maximal conformal CMC spacelike immersion 
from $\Sigma $ into $\LL^3$ can be constructed in this manner, 
using a holomorphic potential $\xihat$ that is well-defined on $\Sigma$.   
\end{theorem}
We call $\phihat$ a \emph{holomorphic extended frame} for the
family of surfaces $f^\lambda$.
It is also true that if we normalize the factors 
in (\ref{thmsplit}) so that $\hat B \in \uhat$,
and define the function $\rho: \Sigma^\circ \to \real$ by $\hat B|_{\lambda=0} = \textup{diag}(\rho, \rho^{-1})$, then there exist conformal coordinates 
$\tilde z= \tilde x+i \tilde y$ on $\Sigma$
such that the induced metric for  $f^1$ is given by
\bdm
\dd s^2 = 4 \rho^4 (\dd \tilde x^2 + \dd \tilde y^2),
\edm
and the Hopf differential is given by $Q \dd \tilde z^2$, where $Q = -2H\frac{b_{-1}}{a_{-1}}$.




\subsection{Behaviour of  the surface at the boundary of the big cell}

Theorem \ref{dpwthm} says that  a standard holomorphic potential
$\xihat$ corresponds to an $H$-surface, 
 provided we restrict to $\Sigma^\circ = \phihat^{-1}(\bigcell)$.
Now set 
\bdm
\mathcal{C}:= \Sigma \setminus \Sigma^\circ = \bigcup_{j=1}^\infty \phihat^{-1} (\mathcal{P}_j),
\hspace{1cm}
\mathcal{C}_1 := \phihat^{-1} (\mathcal{P}_1),
\hspace{1cm}
\mathcal{C}_2 := \phihat^{-1} (\mathcal{P}_2).
\edm

\begin{theorem} \cite{brs} \label{summarythm1}
Let $\phihat$ be as defined in Theorem \ref{dpwthm}.
 Then
\begin{enumerate}
\item  \label{summary0}
 $\Sigma^\circ$ is open and dense in $\Sigma$. 
More precisely, its complement, the set $\mathcal{C}$, 
is locally given as the zero set 
 of a non-constant real analytic function $\Sigma \to \C$.
\item  \label{summary1}
The sets $\Sigma^\circ \cup \mathcal{C}_1$ and $\Sigma^\circ \cup \mathcal{C}_2$ 
are both {open} subsets of $\Sigma$.
The sets $\mathcal{C}_1$ and $\mathcal{C}_2$ are each locally given as the zero set 
 of a non-constant real analytic function $\Sigma  \to \real$.   
\item \label{summary2}
All components of any matrix $F$ obtained by Theorem \ref{dpwthm} on $\Sigma^\circ$,
and evaluated at $\lambda_0 \in \SSS^1$, blow up as $z$ approaches a point $z_0$ in
either $\mathcal{C}_1$ or $\mathcal{C}_2$. In the limit, the unit normal vector $N$, to the corresponding
surface, becomes asymptotically lightlike, i.e. its length in the Euclidean space
$\real^3$ metric approaches infinity.
\item \label{summary3}
The surface $f^{\lambda_0}$ obtained from Theorem \ref{dpwthm}
 extends to a real analytic map 
$\Sigma^\circ \cup \mathcal{C}_1 \to \LL^3$, but 
is not immersed at points $z_0  \in \mathcal{C}_1$.
\item \label{summary4}
The surface $f^{\lambda_0}$ diverges to  $\infty$ as
$z \to  z_0 \in \mathcal{C}_2$. Moreover, the induced metric on the surface blows
up as such a point in the coordinate domain is approached.
\end{enumerate}

\end{theorem}


The arguments given in \cite{brs} to prove those parts of the above theorem
involving $\mathcal{C}_1$ and $\mathcal{C}_2$ all depend on an explicit Iwasawa factorization of
an element of the form $B \omega_1$, where $B$ is an arbitrary
element of  $\uu^\C_+$.  We will use this explicit factorization again
several times below, and so we recall it here:

\begin{lemma} \cite{brs} \label{switchlemma}
Let $\bhat = \begin{pmatrix} a & b \\ c & d \end{pmatrix} = 
\bbar  \rho & 0 \\ 0 & \rho^{-1} \ebar + \bbar  0 & \mu \\ \nu & 0 \ebar \lambda + o(\lambda^2)$ 
be any element of $\uu^\C_+$. Then there 
exists a factorization
\beq \label{switchfact}
\bhat \omega_1 = \hat X \bhat^\prime,
\eeq
where $\bhat^\prime \in \uu^\C_+$ and $\hat X$ is of one of 
the following three
forms: 
\bdm
k_1 = \begin{pmatrix} u & v \lambda \\ \bar{v} \lambda^{-1} & \bar{u} 
\end{pmatrix}, \hspace{.5cm}
k_2 =\begin{pmatrix} u & v \lambda \\ -\bar{v} \lambda^{-1} & -\bar{u} 
\end{pmatrix}, \hspace{.5cm}
\omega_1^\theta = \begin{pmatrix} 1 & 0 \\ e^{i\theta} \lambda^{-1} & 1 
\end{pmatrix}, 
\edm
where $u$ and $v$ are constant in $\lambda$ and can be chosen so 
that the matrix
has determinant one, and $\theta \in \real$. 
The matrices $k_1$ and $k_2$ are in $\uu$, 
and their components  satisfy the equation 
\beq \label{uveqn}
\frac{|u|}{|v|} = |\mu + \rho||\rho| \; . \\
\eeq
The first two forms occurs when $\bhat \omega_1$ is in the big cell $\bigcell$, and the third form occurs if and only if
$\bhat \omega_1$ is in the first small cell, $\mathcal{P}_1$.
The three cases  correspond 
to the cases $|(\mu+ \rho)\rho|$
greater than, less than or equal to 1, respectively.
Moreover, if $\bhat \omega_1$ is given locally by  a
 real analytic map either from $\real^2 \to \bigcell$, or 
 from $\real \to \mathcal{P}_1$,
then the factors $\hat X$ and $\bhat^\prime$ can be chosen to be real analytic.
\end{lemma}
\begin{proof}

One can write down  explicit expressions as follows:
 for the cases 
 $|(\mu+ \rho)\rho|^{\varepsilon} >1$, where $\varepsilon = \pm 1$,
  the factorization is given by
\beq \label{explicitfact}
\begin{split}
&\hat X = 
\begin{pmatrix} u & v \lambda \\ \varepsilon \bar{v} \lambda^{-1} & \varepsilon \bar{u} 
\end{pmatrix} ,  \\
& \bhat^\prime = 
\begin{pmatrix} \varepsilon \bar{u}b \lambda^{-1} - dv + \varepsilon \bar{u}a - vc \lambda ~&~
   b \varepsilon \bar{u} - v d \lambda \\
   -\varepsilon \bar{v} b \lambda^{-2} + (-\varepsilon \bar{v} a + u d)\lambda^{-1} + u c ~&~
     - b \varepsilon \bar{v}\lambda^{-1} + u d \end{pmatrix}. 
\end{split}
\eeq
One can choose $u$ and $v$ so that $\varepsilon (u \bar{u} - v \bar{v}) = 1$ and 
such that $\bhat^\prime \in \uu^\C_+$, the latter condition being 
assured by the requirement that 
$\frac{u}{\bar{v}} = \varepsilon (\mu + \rho)\rho$. Once such choice is
\beq  \label{explicituandv}
v = \frac{1}{\sqrt{ \varepsilon \left( \left| \mu + \rho \right|^2 \left|\rho \right|^2 -1\right)}}, 
\hspace{1cm} u = \varepsilon(\mu+ \rho)\rho \, \bar v.
\eeq
It is straightforward to verify that $\hat X \bhat^\prime =  \bhat \omega_1^{-1}$.

For the case $|(\mu+ \rho)\rho| =1$, use
\beq \label{explicitfact2}
\begin{split}
&\hat X = 
\begin{pmatrix} u & v \lambda \\ - \bar{v} \lambda^{-1} &  \bar{u} 
\end{pmatrix} ,  \\
&\bhat^\prime = 
\begin{pmatrix}  \bar{u}b \lambda^{-1} - dv +  \bar{u}a - vc \lambda ~&~
   b  \bar{u} - v d \lambda \\
    \bar{v} b \lambda^{-2} + (\bar{v} a + u d)\lambda^{-1} + u c ~&~
      b  \bar{v}\lambda^{-1} + u d \end{pmatrix}. 
\end{split}
\eeq
 and choose
$\frac{u}{\bar{v}} =  -(\mu+ \rho) \rho$.
One can choose $u= \frac{1}{\sqrt{2}}$ and $\bar v = \frac{-1}{\sqrt{2}}((\mu + \rho) \rho)^{-1} =
\frac{-1}{\sqrt{2}}e^{i \theta} $ and 
\bdm
\begin{pmatrix} u & v \lambda \\ -\bar{v} \lambda^{-1} & \bar{u} 
\end{pmatrix}
= \begin{pmatrix} 1 & 0 \\ e^{i \theta} \lambda^{-1} & 1 
\end{pmatrix} 
   \begin{pmatrix} \frac{1}{\sqrt{2}} & -\frac{1}{\sqrt{2}} 
e^{- i \theta} \lambda \\ 0 & \sqrt{2} \end{pmatrix}.
\edm
Pushing the  last factor into $\bhat^\prime$ 
then gives the required factorization.  
In this case, $\bhat \omega_1^{-1}$ is in $\mathcal{P}_1$, 
because it can be expressed as 
\bdm
\small\bbar e^{-i\theta/2}& 0 \\ 0 & e^{i\theta/2} \ebar \cdot
  \omega_1  \cdot
\bbar e^{i\theta/2} & 0 \\ 0 & e^{-i\theta/2} \ebar \bhat^\prime.
\edm 
The claimed analytic properties of the factors are satisfied for the explicit 
choices of $u$ and $v$ given above, because the 
expression $(\mu+ \rho) \rho$
is real analytic.
\end{proof}

\section{The Weierstrass representation for surfaces with singularities}
Theorem \ref{summarythm1} states that singularities occur at points which are
mapped into $\mathcal{P}_1$, and that the frame $F$ is not defined at such 
points.  In this section we define an alternative extended frame $\fomega$
which does not blow up at singular points.  This will be used in the next
section to solve the singular Bj\"orling problem.

Let $\pi: \bigcell \to  \uu/\uu^0$ denote the projection
 defined by taking the equivalence class of $\hat F$ (under right multiplication 
by elements of $\uu^0$) in the Iwasawa factorization $\phihat = \hat F \hat B$ 
of $\phihat \in \bigcell$.  Since the Sym-Bobenko formula $\sym$ is invariant
under right multiplication by constant diagonal matrices, $\sym: \uu/\uu^0 \to Lie(\uu)$
is well defined, and we can extended it to a map 
\bdm
\stilde: \bigcell \to Lie(\uu), \hspace{1cm} \stilde = \sym \circ \pi.
\edm
Again we define the map $\stilde_\lambda: \bigcell \to \LL^3$ by
evaluating this at $\lambda \in \SSS^1$.
The crucial fact that is exploited here and 
 in \cite{brs} -- and is proved using Lemma \ref{switchlemma} -- is
  that if $\phihat \in \bigcell$
and $\phihat \omega_1^{-1} \in \bigcell$ then 
\beq \label{equivsym}
\stilde \left(\phihat \, \omega_1^{-1}\right) = \stilde \left(\phihat \right).
\eeq
Thus, if $\phihat: \Sigma \to \uu^\C$, and $\phihat(z_0) = \omega_1 \in \mathcal{P}_1$,
then we can just as well consider the map $\phiom := \phihat \omega_1^{-1}$.
Then $\phiom(z) \in \bigcell$ in a neighbourhood of $z_0$,
and if $\phihat$ is a holomorphic extended frame,
then so is $\phiom$ -- for the same family of surfaces $f^\lambda$.
On the open dense set $\phihat^{-1}(\bigcell) \, \cap \, \phiom ^{-1}(\bigcell)$, we have $\stilde (\phihat) = \stilde (\phiom)$,
and so it is valid to define 
\bdm
f^{\lambda_0}(z_0) :=  \stilde_{\lambda_0} (\phiom (z_0)). 
\edm
Any element of $\mathcal{P}_1$ is of the form $\hat F_0 \omega_1 \hat B_0$,
and essentially the same argument can be used to define $f^{\lambda_0}(z_0)$
when $\phihat(z_0)$ has this form.
Hence one can define a real analytic map
 $f^{\lambda_0}: \phihat^{-1}(\bigcell \cup \mathcal{P}_1) \to \LL^3$ 
 which is an immersed CMC $H$ surface on $\phihat^{-1}(\bigcell)$.

\begin{definition}
Let $\Sigma$ be a simply-connected Riemann surface,
$\xihat$ a standard potential, and $\phihat: \Sigma \to \uu^\C$ the map
obtained by integrating $\phihat^{-1} \dd \phihat = \xihat$ with an initial 
condition $\phihat (z_0) = \phihat_0 \in \uu^\C$.  Assume that $\phihat(w) \in \bigcell$ 
for at least one point $w \in \Sigma$.  Let $\Sigma_s \subset \Sigma$ be the open dense subset
  given by $\Sigma_s = \phihat^{-1} (\bigcell \cup \mathcal{P}_1)$,
and define, for any $\lambda \in \SSS^1$,
\bdm
f^\lambda: \Sigma_s \to \LL^3, \hspace{1cm}
f^\lambda(z) =  \stilde_\lambda \left(\phihat(z)\right).
\edm
We call the map $f^\lambda$ -- and, more generally, any map from a Riemann surface
into $\LL^3$ which has such a representation locally -- a \emph{generalized constant mean curvature $H$  surface}, or \emph{generalized $H$-surface}, in $\LL^3$.
\end{definition}

\subsection{Singular holomorphic potentials and frames} \label{singularholsect}

For a typical generalized  $H$-surface we can expect, from Theorem \ref{summarythm1} Item
\ref{summary1}, that the singular set $\mathcal{C}_1 = \phihat^{-1}(\mathcal{P}_1)$ is a curve,
and we can deduce  from Item \ref{summary2} that this curve must be a null curve, wherever
it is regular.

It is clear from the preceding discussion that one may construct a generalized
 $H$-surface with a singularity
at $z_0$ by integrating a standard potential $\xihat$ with the initial condition
$\phihat (z_0) = \omega_1$, \emph{provided} that the resulting complex extended frame $\phihat$ 
does satisfy $\phihat(z) \in \bigcell$ for some $z$.   Alternatively, supposing
we did this, there is also the translated map $\phiom = \phihat \,  \omega_1^{-1}$ --
which may be more natural because $\phiom(z_0) = I$ and so this maps a neighbourhood
of $z_0$ into the big cell. 

 We first analyze the Maurer-Cartan form of $\phiom$,
given that $\xihat$ is a standard potential, which has the general form:
\beq  \label{phihatform}
\phihat^{-1} \dd \phihat = \left\{ \bbar 0 & a_{-1}  \\ b_{-1} & 0 \ebar \lambda^{-1} + 
    \bbar c_0 & 0 \\ 0 & -c_0 \ebar +
      \bbar 0 & a_1 \\ b_1 & 0 \ebar \lambda  + o(\lambda^2) \right\} \dd z,
\eeq
where $a_{-1}$ is non-vanishing.
For $\phiom = \phihat \,  \omega_1^{-1}$, the above expression is equivalent to
\beqas  
\phiom^{-1} \dd \phiom &=& \left\{ \bbar 0 & 0 \\ -a_{-1} & 0 \ebar \lambda^{-3}
     + \bbar -a_{-1} & 0 \\ 0 & a_{-1} \ebar \lambda^{-2} 
        + \bbar 0 & a_{-1} \\ b_{-1} + 2c_0 - a_1 & 0 \ebar \lambda^{-1} \right.  \\
  &&      
      \left. + \bbar c_0 - a_1 & 0 \\ 0 & -c_0 + a_1 \ebar + 
        \bbar 0 & a_1 \\ b_1 & 0 \ebar \lambda + o(\lambda^2) \right\} \, \dd z. 
\eeqas

Now consider the special case that $\phiom(z) \in \uu$ for $z \in \real$.  Then the Iwasawa factorization of
$\phiom$ along $\real$, is just $\phiom = \phiom \cdot I$, and therefore the Iwasawa
factorization of $\phihat$ for $z \in \real$ is just $\phihat = \phihat_\omega \cdot \,  \omega_1 \, \cdot I$.  In other words, such a holomorphic
frame maps the real line into $\mathcal{P}_1$.  The assumption is equivalent
to demanding 
that $\phiom^{-1} \, \frac{\partial \phiom}{\partial x}(x,0) \, \dd x $ has coefficients in $Lie(\uu)$, 
which implies that it must be of the form:
\beqa   \label{singularreal}
\xihat_0 &=&   \left\{ \bbar 0 & 0 \\ -a & 0 \ebar \lambda^{-3} + 
   \bbar -a & 0 \\ 0 & a \ebar \lambda^{-2} + \bbar 0 & a \\ b & 0 \ebar \lambda^{-1} 
   + \bbar i r & 0 \\ 0 & -i r \ebar \right.\\   
  &&  \left.  + \bbar 0 &  \bar b \\ \bar a & 0 \ebar \lambda
     + \bbar \bar a & 0 \\ 0 & -\bar a \ebar \lambda^2  + 
       \bbar 0 & -\bar a \\ 0 & 0 \ebar \lambda^3 \right\} \dd x, \nonumber  
\eeqa
where $a$ and $b$ are  maps  $\real \to \C$ while $r: \real \to \real$,
and all functions are restrictions to $\real$ of holomorphic functions.
Hence, the Maurer-Cartan form of $\phiom$ is a holomorphic
extension of this:

\begin{definition} \label{singularholdef}
Let   $\Sigma \subset \C$ be a simply connected open subset which
intersects the real line in an interval: $\Sigma \cap \real = J = (x_0, x_1)$,
and contains the origin $z=0$.
A \emph{standard singular holomorphic potential} on $\Sigma$, is a holomorphic 1-form $\xiom$
on $\Sigma$ that can be expressed as:
\beqa   \label{singularhol}
\xihat_\omega = \phiom^{-1} \dd \phiom &=&   \left\{ \bbar 0 & 0 \\ -a & 0 \ebar \lambda^{-3} + 
   \bbar -a & 0 \\ 0 & a \ebar \lambda^{-2} + \bbar 0 & a \\ b & 0 \ebar \lambda^{-1} 
   + \bbar i r & 0 \\ 0 & -i r \ebar \right.\\   
  &&  \left.  + \bbar 0 &  \tilde b \\ \tilde a & 0 \ebar \lambda
     + \bbar \tilde a & 0 \\ 0 & -\tilde a \ebar \lambda^2  + 
       \bbar 0 & -\tilde a \\ 0 & 0 \ebar \lambda^3 \right\} \dd z, \nonumber  
\eeqa
where $a$, $b$ and $r$ are  holomorphic on  $\Sigma$, the restriction of $r$ to $J$
is real, that is $\overline{r(\bar{z})} = r(z)$, and
$\tilde a$ and $\tilde b$ are holomorphic extensions of the restrictions $\bar a \big|_\real$ and
$\bar b \big|_\real$, that is $\tilde a(z)= \overline{a(\bar z)}$, and $\tilde b(z)= \overline{b(\bar z)}$, with the regularity condition:
\begin{itemize}
\item[(A)] \label{conditiona}
$a(z)$ non-vanishing on $\Sigma$.
\end{itemize}\end{definition}

Define the \emph{singular holomorphic frame} $\phiom$ corresponding to $\xihat_\omega$
to be the map $\phiom: \Sigma \to \uu^\C$ obtained by solving the equation

\bdm
\phiom^{-1} \dd \phiom = \xihat_\omega, \hspace{1cm} \phiom(0) = I.
\edm
Set  
\beqas
\phihat \, &:=&  \, \phiom \, \omega_1,\\ 
\Sigma^\circ := \phihat^{-1}(\bigcell), \hspace{.5cm} &&
C := \phihat^{-1}(\mathcal{P}_1), \hspace{1cm}
\Sigma_s := \Sigma^\circ \cup C.
\eeqas
Note that 
$\phihat(0) = \omega_1 \notin \bigcell$ so it is not clear
that  $\Sigma^\circ$ is non-empty.

\begin{theorem}  \label{thm1}
Suppose $\xihat_\omega$ is a standard singular holomorphic potential given by Definition
\ref{singularholdef}, and suppose that $\Sigma^\circ$ is non-empty.
Then 
\begin{enumerate}
\item \label{thm1item1}
 $\Sigma^\circ$ is
open and dense in $\Sigma$.
\item  \label{thm1item2}
$\Sigma_s$ is also an open dense subset of $\Sigma$. For any $\lambda \in \SSS^1$,
the map $f^\lambda: \Sigma_s \to \LL^3$, given by 
\beqas
f^\lambda &=& \stilde_\lambda \left(\, \phiom \, \right) \\
&=& \stilde_\lambda \left(\, \phihat \, \right),
\eeqas
is a generalized constant mean curvature $H$ surface.  
\item \label{thm1item3}
The restriction $f^\lambda \, \big|_{\Sigma^\circ}: \Sigma^\circ \to \LL^3$ is a
spacelike CMC $H$ immersion.
\item  \label{thm1item4}
The map $f^\lambda$ is not immersed at points $z \in C$, and 
the interval $J = \ \Sigma \cap \real$ is contained in the singular set $C$.
Moreover,  $f^\lambda \big |_J$ is 
either a single point or a real analytic
null curve which is regular except at points where $\Re (a \lambda^{-2}) = 0$.

\item \label{thm1item5}
A condition that ensures that $\Sigma^\circ$ is non-empty is:
\begin{itemize}
\item[(B)] \label{conditionb}
 $r - \Im b$   not equivalent to zero on $J = \Sigma \cap \real$.
 \end{itemize}
 Moreover, on a neighbourhood in $\Sigma$ of a point $z_0 \in J$, such that
 $r(z_0)-\Im b(z_0) \neq 0$, the sets $C$ and  $J$ coincide.
\end{enumerate}
\end{theorem}

\begin{proof}
\noindent \textbf{Items \ref{thm1item1}-\ref{thm1item3}:}
The Maurer-Cartan form of $\phihat = \phiom \omega_1$ is given by
\beq \label{phiinvdphi}
\phihat^{-1} \dd \phihat = \bbar ir + \tilde b  &
     ~~ a \lambda^{-1} + \tilde b \lambda - \tilde a \lambda^3 \\
      2i \left( \frac{1}{2i}(b-\tilde b) - r\right) \lambda^{-1}  ~~ &
       -ir - \tilde b  \ebar \dd z,
\eeq
and we assumed $a$ is non-vanishing, so this 
is a standard holomorphic potential.
Since $\xihat_\omega$ is $Lie(\uu)$-valued along $\real$, it follows that 
$\phiom$ maps $J \subset \real$ into $\uu$.  Therefore $\phihat= \phiom \, \omega_1$
maps $J$ into $\mathcal{P}_1$, by definition of $\mathcal{P}_1$.
Hence items \ref{thm1item1}-\ref{thm1item3}
follow from Theorem \ref{summarythm1} and equation (\ref{equivsym}) above.

\noindent \textbf{Item \ref{thm1item4}:}
The first statement follows from Theorem \ref{summarythm1}, so we 
are left with the second statement concerning the regularity of
 $f^\lambda \big|_J$.
 
   First, since $\phiom(z) \in \uu \subset \bigcell$ for real values of $z$,
 it follows that the set $W = \phiom^{-1} (\bigcell)$ is open (and, in 
 fact dense, see the proof of Theorem 4.1 of \cite{brs}) and  contains $J$.  Hence, pointwise on this set, we can decompose
 \beqas
\phiom = \hat F_\omega  \hat B_\omega, \hspace{1cm} \hat F_\omega \in \uu, \hspace{.5cm} \hat B_\omega \in \uhat\\
\hat F_\omega \, |_J = \phiom \, |_J, \hspace{1cm} \hat B_\omega \, |_J = I.
\eeqas
We will call $\hat F_\omega$ a \emph{singular frame} for $f^\lambda$.
Since $\hat B_\omega$ is normalized, the factors $\hat F_\omega$ and $\hat B_\omega$ depend
real analytically on $z$, and we can write
\bdm
\hat B_\omega = \bbar \rho & 0 \\ 0 & \rho^{-1} \ebar + \bbar 0 & \mu   \\ \nu & 0 \ebar \lambda + o(\lambda^2),
\edm
where $\rho$ is a positive real valued function, and $\mu$ 
and $\nu$ are $\C$-valued.
Now on $W$, we have $\phihat = \hat F_\omega \hat B_\omega \, \omega_1$, 
and
since $\hat B_\omega = I$ along $J$, we have, for $z \in J$,
\beqas
\hat F_\omega ^{-1} \dd \hat F_\omega &=& \phiom^{-1} \dd \phiom - \dd \hat B_\omega\\
  &=& \xihat_\omega \, - \,  \bbar \dd \rho & 0 \\ 0 & -\rho^{-2} \dd \rho \ebar -
   \bbar 0 & \dd \mu \\ \dd \nu & 0 \ebar \lambda  + o(\lambda^2).
\eeqas
Because $\hat F_\omega$ is $\uu$-valued, it now follows from equation (\ref{singularhol}) 
and the reality condition defining $\uu$
that, for $z \in J$,
\beqas
\hat F_\omega ^{-1} \dd \hat F_\omega &=& 
   \left\{ \bbar 0 & 0 \\ -a & 0 \ebar \lambda^{-3} + 
   \bbar -a & 0 \\ 0 & a \ebar \lambda^{-2} + \bbar 0 & a \\ b & 0 \ebar \lambda^{-1} \right\} \dd z \\
     && 
   + \bbar i r & 0 \\ 0 & -i r \ebar \dd z -  \bbar \dd \rho & 0 \\ 0 & -\rho^{-2} \dd \rho \ebar\\   
  &&  + \left\{ \bbar 0 &  \bar b \\ \bar a & 0 \ebar \lambda
     + \bbar \bar a & 0 \\ 0 & -\bar a \ebar \lambda^2  + 
       \bbar 0 & -\bar a \\ 0 & 0 \ebar \lambda^3 \right\} \dd \bar z,  
\eeqas
and it is necessary that
\bdm
\bbar 0 &  \bar b \\ \bar a & 0 \ebar \lambda \dd z - 
   \bbar 0 & \dd \mu \\ \dd \nu & 0 \ebar \lambda = 
          \bbar 0 &  \bar b \\ \bar a & 0 \ebar \lambda \dd \bar z.
\edm
The (1,2) component of this matrix equation is equivalent to
\bdm
\mu_x =0, \hspace{1cm} \mu_y = 2i \bar b.
\edm
The reality condition for $\hat F_\omega^{-1} \dd \hat F_\omega$ also requires that
the (1,1) component of the term constant in $\lambda$ is pure imaginary, so
\bdm
ir(\dd x + i \dd y) - \rho_x \dd x - \rho_y \dd y = i (p \dd x + q \dd y),
\edm
for some real functions $p$ and $q$.  The real part of this equation is
equivalent to
\bdm
\rho_x = 0, \hspace{1cm} \rho_y = -r.
\edm

Writing the $(1,1)$ term as $ir \dd z -(-r) \dd y = \frac{ir}{2} \dd z + \frac{ir}{2} \dd \bar z$,
we have  just seen that, along $J$, the singular frame has Maurer-Cartan form:
\beqa
\hat F_\omega^{-1} \dd \hat F_\omega &=& \hat U _\omega \dd z + \hat V _\omega  \dd \bar z, \nonumber \\
 \hat U _\omega = \bbar -a \lambda^{-2} + \frac{ir}{2} & a \lambda^{-1} \\
            -a \lambda^{-3} + b \lambda^{-1} & a \lambda^{-2} -   \frac{ir}{2}   \ebar,
            &&
  \hat V _\omega = \bbar  \frac{ir}{2} + \bar a \lambda^2 & \bar b \lambda -
              \bar a \lambda^3 \\      \bar a \lambda &
               - \frac{ir}{2}- \bar a \lambda^2 \ebar. \label{uhatandvhat}
\eeqa
Differentiating the Sym-Bobenko formula (\ref{symformula}), we obtain
\beqas
\hat F_\omega^{-1} f^\lambda_z \, \hat F_\omega &=& -\frac{1}{2H} \left( [ \hat U _\omega , e_3 ] + 2 i\lambda \frac{\partial}{\partial \lambda} \hat U _\omega  \right),\\
  &=& -\frac{2 ia \lambda^{-2}}{H}  \bbar 1 & - \lambda \\ \lambda^{-1} & -1 \ebar,
\eeqas
and similarly,
\bdm
\hat F_\omega^{-1} f^\lambda_{\bar z} \, \hat F_\omega = \frac{-2i \bar a \lambda^2}{H} 
   \bbar 1 & - \lambda \\ \lambda^{-1} & -1 \ebar.
\edm
Adding and subtracting these equations leads to
\beqa  \label{singframeeqns}
\hat F_\omega^{-1} f^\lambda_{x} \, \hat F_\omega & = &  \frac{-4 \Re (a \lambda^{-2})}{H}
    \bbar i & -i \lambda \\ i \lambda^{-1} & -i \ebar, \\
 \hat F_\omega^{-1} f^\lambda_{y} \, \hat F_\omega & = &  \frac{4 \Im (a \lambda^{-2})}{H}
    \bbar i & -i \lambda \\ i \lambda^{-1} & -i \ebar. \nonumber
 \eeqa
Now, since $\hat F_\omega(z, \bar z, \lambda)$ is an element of  $SU_{1,1}$, it acts by isometries
on $\mathfrak{su}_{1,1} = \LL^3$, and it follows that $f^\lambda_x$ and 
$f^\lambda_y$ are parallel and null. Moreover, $f^\lambda_x \in \LL^3$ is the
zero vector if and only if $\Re(a \lambda^{-2}) = 0$.   Since $a$ is holomorphic,
either the real part of $a \lambda^{-2}$ is  equivalent to zero along the
real line, in which case
$f^\lambda (J)$ is a single point, or $\Re (a \lambda^{-2})$ 
 has isolated zeros on $J$, and $f^\lambda \big|_J$ is regular
 away from these zeros.
 
\noindent \textbf{Item \ref{thm1item5}:}
By Lemma \ref{switchlemma}, $\phihat$ is in the big 
cell if and only if 
\beq  \label{absinequality}
h \,:= \, \left|\mu + \rho\right|^2 \left| \rho \right|^2 -1 \neq 0.
\eeq
Now we know that for $z \in J$, we have $\rho = 1$ and $\mu= 0$,
so $h= |\mu + \rho|^2 \, | \rho|^2 -1 = 0$ along $J$ as expected.
To guarantee that $\Sigma^\circ$ is non-empty, we need 
 to ensure that $h$ is not constant, and for this it is sufficient to require
that $\frac{\partial h}{\partial y} \neq 0$ at at least one point $z \in J$.
Using the above expressions for $\rho_y$ and $\mu_y$, and $\rho=1$, $\mu=0$,
one computes
\beqas
\frac{\partial h}{\partial y} &=& 4 \rho_y + (\mu_y + \bar \mu_y) \\
  &=& -4r + 4 \Im b.
\eeqas
If this expression is non-zero at 
$z_0 \in J$, then it is also non-zero on a neighbourhood $\mathcal{N}$ of $z_0$,
and, because $h=0$ and $h_y \neq 0$ on $J\cap\mathcal{N}$  it follows that, taking $\mathcal{N}$ smaller if necessary, the zero set 
$C \cap \mathcal{N}$ of $h \big|_{\mathcal{N}}$ is precisely
$J \cap \mathcal{N}$.
\end{proof}


\textbf{Note:} From here on, to simplify notation, we consider mainly $f = f^1$, 
rather than $f^{\lambda_0}$ for other values of $\lambda_0 \in \SSS^1$.  We will also use
the convention $X := \hat X \big|_{\lambda =1}$, if $\hat X$ depends on $\lambda$.

One  has the following formulae for the metric and Hopf differential of the surface just constructed:

\begin{lemma} \label{metriclemma}
Let $f = \stilde_1 (\phiom) = \stilde_1 (\phihat)  : \Sigma_s \to \LL^3$ be a generalized $H$-surface constructed from a 
singular holomorphic frame, factored on $\phiom^{-1}(\bigcell)$ as $\phiom = \fomega \hat B_\omega$ as in 
 Theorem \ref{thm1}, and write the Fourier expansion of
  the matrix valued function 
  $\hat B_\omega \in \uhat$ as:
  \bdm
\hat B_\omega = \bbar \rho & 0 \\ 0 & \rho^{-1} \ebar + \bbar 0 & \mu   \\ \nu & 0 \ebar \lambda + o(\lambda^2).
\edm

Let $\Sigma^\pm:= \phihat^{-1}(\bigcell^\pm)$. Then:  
\begin{enumerate}
\item  \label{lemmaitem1}
The  metric $\dd s^2$, induced by $f$  on
 $\phiom^{-1}(\bigcell)$, is given 
by the formula 
\beqa \label{metric}
\dd s^2 = 4 g^2 \,\, (\dd x^2 + \dd y^2), \hspace{1cm} 
     g = \varepsilon e^{u} =  \varepsilon \frac{\chi^2 |a|}{H}, \\
     \varepsilon(z) = \pm 1, \,\,\,\textup{for } z \in \Sigma^{\pm}, 
     \hspace{1cm}
\chi = \sqrt{\left| |\mu + \rho|^2 - \rho^{-2} \right|}. \label{chidef}
\eeqa

The function $g$ is real analytic on $\phiom^{-1}(\bigcell) \setminus \real$,
and extends as a $C^1$ function across the real line.
It has the following  values at a point  
$z_0 \in \real \cap  \phiom^{-1}(\bigcell)$:
\beq \label{realgequations}
 g = 0, \hspace{1cm} 
 \frac{\partial g}{\partial x} = 0, \hspace{1cm}
\frac{\partial g}{\partial y} = \frac{4 |a|(\Im b - r)}{H}.
\eeq
\vspace{2ex}

\item \label{lemmaitem2}
The Hopf differential on $\phiom^{-1}(\bigcell)$ is given by $Q\dd z$, where
\beq \label{hopfdiff}
Q = \frac{2a}{H}(b-\tilde b - 2 i r).
\eeq
\end{enumerate}
\end{lemma}
\begin{proof}
\noindent \textbf{Item \ref{lemmaitem1}:}
On $\phiom^{-1}(\bigcell) \cap \phihat^{-1}(\bigcell)$ we have, using Lemma \ref{switchlemma},
\beqas
\phihat &=& \fomega \hat B_\omega \omega_1 = \fomega \hat X \hat B^\prime \\
&=& \hat F \hat B,     
\eeqas
where $\fhat = \varepsilon \fomega \hat X$, $\bhat = \varepsilon \bhat^\prime$,
and  $X$ and $\bhat^\prime$ are given
 in equation (\ref{explicitfact}).
Writing the Fourier expansion
\bdm
\hat B = \bbar \chi & 0 \\ 0 &  \chi^{-1} \ebar + o(\lambda),  
\edm
the choice of $u$ and $v$ in $\bhat^\prime$ given in Lemma \ref{switchlemma} gives 
the formula (\ref{chidef}) for $\chi$. Since $\chi >0$, this is the unique Iwasawa factorization
$\phihat = \fhat \bhat$ with $\bhat \in \uhat$.

Using this, and the expression (\ref{phiinvdphi}) for $\phihat^{-1} \dd \phihat$,
 one obtains
\beqas
\hat F^{-1} \dd \hat F &=& \hat B \phihat^{-1} \dd \phihat \hat B^{-1} + \hat B \dd \hat B^{-1}\\
 &=& \bbar 0 & \chi^2 a \, \lambda^{-1} \\ 
 \chi^{-2}(b-\tilde b - 2i r) \, \lambda^{-1} & 0 \ebar   \dd z 
 + o(\lambda^0).
 \eeqas

To calculate the metric, the formulae (\ref{derivatives}), at $\lambda=1$, for $f_z$ and $f_{\bar z}$ then give:
\beqas
f_x &=&  \frac{2i}{H} \, F  \,  \bbar 0 & \chi^2 a \\ 
    - \chi^2 \bar a & 0 \ebar \,  F^{-1} \\
&=&
    \frac{2\chi^2 |a|}{H} \,  F_C   \, e_1 \,  F_C^{-1},
\eeqas
where
\beq \label{coordframe}
 \fhat_C := \fhat  D, \hspace{1cm}
  D =  \bbar e^{i(\frac{\phi}{2} + \frac{\pi}{4})} &0 \\
                             0 &  e^{-i(\frac{\phi}{2} + \frac{\pi}{4})}  \ebar,
           \hspace{1cm}     a = |a|e^{i\phi}.
\eeq 
A well-defined choice for the function $\phi$ can be made
because $a$ is non-vanishing on the
simply connected set $\Sigma$.
Similarly we have
\bdm
     f_y =   \frac{2\chi^2 |a|}{H} \,  F_C   \,  e_2 \,  F_C^{-1}.
\edm
It follows that $\fhat_C$ is the coordinate frame defined by equations 
(\ref{framedefn}) and that 
$2e^u = \frac{2\chi^2 |a|}{H}$ (recalling that we have assumed $H$ is positive), which gives the formula (\ref{metric})
for the metric.  The factor $\varepsilon$ is included  to achieve 
continuity of the derivatives of $g$ across $\real$.

The function $g = \varepsilon \frac{\chi^2 |a|}{H}$ is real analytic everywhere on $\phiom^{-1}(\bigcell) \setminus J$, 
because $\rho$ and $a$ are non-vanishing and $g$ is non-vanishing
 on this set. It has the
limiting value zero for $z \to J$, because $\rho\big|_J = 1$ and $\mu \big|_J =0$.
 To compute the limits of the
derivatives at  (\ref{realgequations}) for real values of $z$, 
one can differentiate the formula
$\chi = \sqrt{\varepsilon \left( |\mu + \rho|^2 - \rho^{-2}\right)}$, 
with $\varepsilon = \pm 1$ for $z \in \Sigma^\pm$,
 and use the equations 
$\mu_x \to 0 = \rho_x \to  0$,  $\mu_y \to 2i \bar b$,  $\rho_y \to -r$, found in the proof of
Theorem \ref{thm1}.

\noindent \textbf{Item \ref{lemmaitem2}:}
The standard coordinate frame  $\hat F_C$, found above, satisfies
\beqas
\hat F_C^{-1} \dd \hat F_C &=& 
  \bbar 0 & -i  \chi^2 |a| \, \lambda^{-1} \\ 
 i \frac{a}{|a|} \chi^{-2}(b-\tilde b - 2i r) \, \lambda^{-1} & 0 \ebar   \dd z + o(1),\\
&=& \hat U \dd z + \hat V \dd \bar z,
\eeqas
where $\hat U$ is given at (\ref{withlambda}).
Comparing the off-diagonal components of the above matrix with those of 
$\hat U$,  and using $\chi^2 = \frac{e^u H}{|a|}$,
we have
\bdm
i\frac{a}{|a|} \frac{|a|}{H e^u} ( b-\tilde b -2 i r) = \frac{1}{2}ie^{-u}Q,
\edm
which is the expression (\ref{hopfdiff}) for $Q$.
\end{proof}


\subsection{The converse of Theorem \ref{thm1}}
Next we show  that every 
generalized $H$-surface that contains a curve in the coordinate
domain of its singular set
can be locally represented, around that curve, by a standard 
singular holomorphic potential.
 
If $\phihat: \Sigma \to \uu^\C$
is a holomorphic map,
 and  $\phihat$  maps at least one point into $\bigcell$, then,
according to Theorem \ref{summarythm1}, the singular set
$C = \phihat^{-1}(\mathcal{P}_1)$ 
   is locally given as the zero
set of a non-constant real analytic function $h: \real^2 \to \real$.
In our setting, $h$ is given by the formula 
 (\ref{absinequality}), $
h \,:= \, \left|\mu + \rho\right|^2 \left| \rho \right|^2 -1$.

\begin{definition} \label{degendef}
A point $z_0  \in \phihat^{-1}(\mathcal{P}_1)$ is said to be a
 \emph{non-degenerate singular point}   
 if  the derivative map $\dd h$ has rank 1 at $z_0$, and \emph{degenerate} if 
 $\dd h = 0$.  If, at a point $z_0 \in \phihat^{-1}(\mathcal{P}_1)$ we have
 the milder condition that there exists a real analytic curve $\gamma: (-\delta,\delta) \to \Sigma$, for some $\delta>0$, with $\gamma(0)= z_0$
 and $\gamma((-\delta,\delta)) \subset \phihat^{-1}(\mathcal{P}_1)$,
then we call $z_0$
  \emph{weakly non-degenerate}.
 A generalized $H$-surface is \emph{non-degenerate} or
 \emph{weakly non-degenerate} if  all singular points have the corresponding property.
 \end{definition}
For a surface constructed via Theorem \ref{thm1}, the non-degeneracy condition
is $\Im b - r \neq 0$.

\begin{theorem} \label{thm2}
Let $f: \Sigma_s \to \LL^3$ be a generalized $H$-surface with a 
corresponding standard potential $\xihat$ and holomorphic extended frame
$\phihat$, with $f = \stilde_1(\phihat)$.
Let $z_0 \in C = \phihat^{-1}(\mathcal{P}_1)$
 be a weakly non-degenerate singular point. 
 Then, on an open set $\Omega \subset \Sigma_s$,
 containing $z_0$,
there exist conformal coordinates and a standard singular holomorphic potential 
$\xihat_\omega$, of the form (\ref{singularhol}), with corresponding
singular holomorphic extended frame $\hat \Psi_\omega$,
 such that $f$ is represented
on $\Omega$ by the surface $\stilde_1(\hat \Psi_\omega)$.
\end{theorem}
\begin{proof}
 If $z_0 \in C$ and $\phihat(z_0) = \hat F_0 \, \omega_1 \hat B_0$ is the Iwasawa factorization,
 set $\phihat_\omega (z) = \phihat (z) \, \hat B_0^{-1} \, \omega_1^{-1}$.
Then
$\phihat_\omega(z_0) = \hat F_0 \in \bigcell$, so locally
we can Iwasawa factorize 
$\phihat_\omega (z) = \hat F_\omega (z) \, \hat B_\omega (z)$,
with the two factors in $\uu$ and $\uhat$ respectively.   Now
\beq  \label{phihatfact}
\phihat (z) = \phiom (z) \,  \omega_1 \hat B_0    = 
  \hat F_\omega (z) \, \hat B_\omega (z)\, \omega_1 \hat B_0,
\eeq
  and this is in the big cell precisely when 
$\hat B_\omega (z) \, \omega_1$ is.
As $z_0$ is weakly non-degenerate, there is a curve through $z_0$
which is mapped by $\phihat$ into $\mathcal{P}_1$.
  After a conformal change of
coordinates (taking a smaller neighbourhood if necessary) we can assume 
that this curve  is an open interval $J$ on
the line $\{y=0\} \subset \C$, and that $z_0$ is the origin. By Lemma \ref{switchlemma}, we can, on the interval $J$, write 
\bdm
\hat B_\omega (x,0)\, \omega_1 = 
R_\theta (x) \, \omega_1   \widetilde{B} (x), \hspace{1cm}
  R_\theta (x) := \bbar e^{-i\theta(x)/2}& 0 \\ 0 & e^{i\theta(x)/2} \ebar \in \uu,
  \hspace{.5cm} \widetilde{B}(x) \in \uu^\C_+,
\edm
where $R_\theta$ and $\widetilde{B}$ are real analytic in $x$.
Substituting into equation (\ref{phihatfact}), this means
\bdm
\phihat \big|_J (x) 
   =  F_*(x) \, \omega_1 \, B_*(x), 
   \hspace{1cm} F_*(x) := \hat F_\omega (x,0) \, R_\theta (x), \hspace{.5cm}
   B_*(x) :=\widetilde{B} (x) \, \hat B_0.
\edm
Now, by extending $\theta(x)$ analytically, 
 $R_\theta$ has a holomorphic extension  $\check R_\theta: \Omega \to \uu^\C$
 to some open set $\Omega$ containing
$I$. Similarly, since the Maurer-Cartan form of $\fomega \big|_J$, 
has only a finite number of real analytic functions in its Fourier expansion
in $\lambda$, this map also
has a holomorphic extension to
a map $\check F_\omega: \Omega \to \uu^\C$, taking $\Omega$ sufficiently small.
Therefore
 $B_* = \omega_1^{-1} \,\cdot  R^{-1}_\theta \, \cdot \, \hat F_\omega^{-1} \big|_J \,
   \cdot \, \phihat \big|_J$ 
 extends holomorphically to  a map $B_*: \Omega \to \uu^\C_+$, given by
 $B_*(z) = \omega_1^{-1} \, \cdot  \check R^{-1}_\theta(z) \, \cdot \, \check F_\omega^{-1}(z) \, \cdot \, \phihat (z)$.  This allows one to define a 
 holomorphic map 
\beqas
\hat \Psi (z) &:=& \phihat (z) \, B_*^{-1} (z)\\
  &=&   \check F_\omega (z) \,  \check R_\theta (z) \, \omega_1.
\eeqas
This has the property that
 $\stilde(\hat \psi (z)) = \stilde(\phihat(z))$, because
$B_*^{-1}(z) \in \uu^\C_+$ and therefore has no impact on the Iwasawa decomposition
of $\phihat$. Moreover, it is easy to verify that 
$\hat \Psi^{-1} \dd \hat \Psi$ is also a \emph{standard} holomorphic
potential, because right multiplication by a holomorphic map into $\uu^\C_+$
preserves the relevant properties.
 Finally, consider the translate,
 $\hat \Psi _\omega := \hat \Psi \omega_1^{-1}$. By definition,
 we have
 \bdm
 \hat \Psi _\omega \big|_J (x) = F_*(x) \,  \in \, \uu.
\edm
 Hence, as shown in Section \ref{singularholsect}, it follows that
$\hat \xi _\omega := \hat \Psi _\omega ^{-1} \dd  \hat \Psi _\omega$
is a singular holomorphic potential of the form given by
(\ref{singularhol}).
By construction, we have, on the open set $\Omega$,
\bdm
\stilde_1 \, (\hat \Psi _\omega)  
  =  \stilde_1 \, (\hat \Psi) 
  = \stilde_1 \, (\phihat)  = f.
\edm
\end{proof}


\section{Prescribing singularities: the singular Bj\"orling problem}  \label{bjorlingsection}
We showed that if $f: \Sigma_s \to \LL^3$
is a generalized $H$-surface, and $z_0 \in \Sigma_s$ is a weakly non-degenerate singular point,
 then, at least locally,
$f$ can be constructed from a singular frame $\hat F_\omega$ which satisfies the
equations (\ref{singframeeqns}), which, at $\lambda=1$, are:
\beq \label{singframe2}
 F_\omega^{-1} f_{x} \,  F_\omega  =  \frac{-4 \Re (a)}{H}
    \, (-e_2 + e_3), 
  \hspace{1cm}
  F_\omega^{-1} f_{y} \,  F_\omega  =   \frac{4 \Im (a)}{H}
    \, (-e_2 + e_3). 
 \eeq
The singular Bj\"orling problem can be stated as the task of 
 constructing the singular 
frame  $\fomega$ -- and hence the surface -- given that we 
only know $f$ (and therefore $f_x$, if $x$ is the parameter of the curve)
 and $f_y$ along the singular curve.
 
So suppose we have an open set $\Omega \subset \C$, with coordinates
$z=x+iy$, and  such that
$J = \Omega \cap \real =  (x_1, x_2)$ is a non-empty open interval containing the
origin. Suppose there exists a 
generalized $H$-surface $f: \Omega \to \LL^3$, satisfying the Bj\"orling data 
along $J$, and with associated holomorphic
extended frame $\phihat$,  such that $\phihat(J) \subset \mathcal{P}_1$.
 Since the vector fields $f_x$ and $f_y$ are both necessarily null 
 and parallel along $J$, we can, on this interval, and after an isometry of $\LL^3$, write
\bdm
f_x  = s \bbar i & e^{i\theta} \\ e^{-i\theta} & -i \ebar, \hspace{1cm}
f_y = t \bbar i & e^{i\theta} \\ e^{-i\theta} & -i \ebar,
\hspace{1cm}
 \theta(0) = -\frac{\pi}{2},
\edm
where $s$, $\theta$ and $t$ are all real analytic
 functions $J \to \real$.  We assume that $s$ and $t$ never vanish at the
 same time,
 so that $\theta$ is well defined on $J$.  
 
The equations (\ref{singframe2})
 suggest that we  choose a frame $F_0$ to be the rotation
 about the $x_3$-axis  which rotates 
 $[\cos \theta, \sin \theta, 0]^T \in \LL^3$ so that it points 
 in the $-e_2$ direction:
 \beq  \label{singcoordframe}
 F_0 = \bbar e^{i\frac{ 2\theta+\pi}{4}} & 0 \\
         0 &  e^{-i\frac{2\theta+\pi}{4}} \ebar.
\eeq
The normalization of $\theta$ means that $F_0(0) = I$. Then 
\beq \label{fxandfy}
F_0^{-1} f_x F_0 = s \, (-e_2 + e_3), \hspace{1cm} 
  F_0^{-1} f_y F_0 = t \, (-e_2 + e_3).
\eeq
Comparing this with equations (\ref{singframe2}), we must have, along $J$,
\bdm
\Re a = -\frac{Hs}{4}, \hspace{1cm} \Im a = \frac{H t}{4}.
\edm
Thus our regularity assumption on $s$ and $t$ is actually equivalent to the
assumption that the surface is a generalized $H$-surface, i.e. $a$ is non-vanishing.

To find the $\lambda$ dependence of the singular frame, we
know from equation (\ref{singularreal}) that this frame satisfies:
\beq \label{fomegamcform}
\fomega^{-1} \dd \fomega =
  \left\{ \bbar -a \lambda^{-2} & a\lambda^{-1} \\ -a \lambda^{-3} +b \lambda^{-1} \,\, & a\lambda^{-2} \ebar 
   + \bbar i r & 0 \\ 0 & -i r \ebar    
    + \bbar \bar a \lambda^2 &  \,\, \bar b \lambda -\bar a \lambda^3  \\ \bar a \lambda & -\bar a \lambda^2 \ebar  \right\} \dd x.
\eeq
Evaluating at $\lambda=1$ and comparing this with the Maurer-Cartan form of our frame:
\bdm
F_0^{-1} \dd F_0 = \bbar \frac{i}{2}\theta_x & 0 \\ 
    0 & - \frac{i}{2}\theta_x \ebar \dd x,
\edm
and using the above formula for $\Im a$, we obtain along $J$ the values : 
$r=  \frac{1}{2} (\theta_x +  H t)$, and  $b=  \frac{1}{2} i  H t$.
Substituting $a$, $b$ and $r$ into equation (\ref{fomegamcform}) and 
extending holomorphically, gives the singular holomorphic potential $\xiom$. The 
non-degeneracy condition $r- \Im b \neq 0$ for the singular curve is 
\beq \label{regcond}
 \theta_x \neq 0.
\eeq


\begin{theorem}  \label{bjorling}
Suppose given a real analytic function $f_0: J\to \LL^3$, such that $\frac{\dd f_0}{\dd x}$ is a null vector field,  and an additional null real analytic vector field $v(x)$, such
that $v(x)$ is a scalar multiple of $\frac{\dd f_0}{\dd x}(x)$ for
each $x \in J$. 
Suppose also that
 the vector fields do not vanish simultaneously
at any point $x \in J$. Let $s$ and $t$ be defined as above.
Let $\phiom$ be the singular holomorphic frame obtained by analytically extending
the 1-form $\fomega^{-1} \dd \fomega$ given by (\ref{fomegamcform}), with
\bdm
a = \frac{H}{4}(-s + i t), \hspace{1cm}
b=  \frac{1}{2} i  H t, \hspace{1cm}r=  \frac{1}{2} (\theta_x +  H t),
\edm
 to some 
simply connected open set
containing $J$, and integrating with initial condition $\phiom (0)=I$.
Suppose that $\phihat = \phiom \omega_1$ maps at least one point into $\bigcell$.
Then
 the  surface 
 \bdm
 f(x,y) := \stilde_1(\phiom(x,y)) + \frac{1}{2H} e_3 + f_0(0),
 \edm
 is the unique weakly non-degenerate generalized $H$-surface
such that $f$, $f_x$ and $f_y$ coincide respectively with $f_0$, $\frac{\dd f_0}{\dd x}$ and $v$ along the real interval $J$.
\end{theorem}
Uniqueness here is understood to mean that the two surfaces are both defined
and agree on some open subset of $\C$ containing the interval $J$.
We remark that a condition  that guarantees that $\phihat$ maps at least one point into the big cell is that $\frac{\dd^2 f_0}{\dd ^2 x}$ is not
parallel to  $\frac{\dd f_0}{\dd x}$ (that is, $\theta_x \neq 0$) at some point on $J$. 

\begin{proof}
By construction, and with the assumption that $\phihat^{-1}(\bigcell)$ is
non-empty,  $f$ is a generalized $H$-surface that has the 
required values along $J$, so we need to show uniqueness.

Suppose $\tilde f$ is another generalized $H$-surface satisfying the 
Bj\"orling data. It is necessarily weakly non-degenerate. 
By Theorem \ref{thm2}, there exists a standard
singular holomorphic potential $\tilde \xi_\omega$ and singular 
holomorphic frame $\hat \Psi_\omega$
such that $\stilde_1(\hat \Psi_\omega) =  \tilde f + \textup{translation}$.
No coordinate
change is necessary, since the condition that $\tilde f$ is not immersed along
$J$ implies that the holomorphic extended frame 
defining  $\tilde f$ already maps $J$ into
$\mathcal{P}_1$.
 
Let $\hat G_\omega$ be the singular frame  obtained by the Iwasawa
decomposition $\hat \Psi_\omega = \hat G_\omega \hat B  _\omega$, with 
$\hat B _\omega \in \uhat$. 
As shown in the proof of Theorem \ref{thm1}, the map 
$\tilde f$ satisfies, at points $z \in J$,
\beq    \label{tildeeqns}
\hat G_\omega^{-1}  \,\, \tilde f_{x} \,\,  \hat G_\omega  =  \frac{-4 \Re (A \lambda^{-2})}{H}
    \bbar i & -i \lambda  \\ i \lambda^{-1}  & -i \ebar, 
  \hspace{0.5cm}
 \hat G_\omega^{-1} \, \, \tilde f_{y} \, \, \hat G_\omega 
    =   \frac{4 \Im (A \lambda^{-2})}{H}
    \bbar i & -i \lambda  \\ i \lambda^{-1}  & -i \ebar,
 \eeq
where $\hat \Psi^{-1} \dd \hat \Psi = \bbar 0 & A \\ B & 0 \ebar \lambda^{-1} \dd z + o(\lambda)$, and $\hat \Psi := \hat \Psi_\omega \omega_1$.
On the other hand, we have, by assumption that $\tilde f_x$ and $\tilde f_y$ satisfy
the equations (\ref{fxandfy}), namely, along $J$,
\bdm
 \tilde f_x  =  s  F_0  \, (-e_2 + e_3) \, F_0^{-1},
\hspace{1cm}
 \tilde f_y  =  t  F_0  \, (-e_2 + e_3) \, F_0^{-1}.
\edm
We will first show that we can assume, without loss of generality, that
$\Re A = -\frac{Hs}{4}$ and  $\Im A = \frac{H t}{4}$ as follows:
comparing the equations above, it follows that, wherever $s\neq 0 \neq t$
we have
\bdm
\frac{t}{\Im A} = \frac{-s}{\Re A} =: \kappa.
\edm
At least one of $s(x)$ or $t(x)$ is non-zero at each $x \in J$, and so 
$\kappa: J \to \real$ is well defined and
non-vanishing.  Let $\beta$ be the holomorphic
extension of $\frac{\sqrt{\kappa H}}{2}$ to a simply connected open set
 $\mathcal{N} \subset \C$ which contains $J$. Set
\bdm
\hat \Psi^\prime := \hat \Psi \bbar \beta^{-1} & 0 \\ 0 & \beta \ebar.
\edm
Then $\stilde (\hat \Psi^\prime) = \stilde (\hat \Psi)$ because
the $\uu$ factor in the Iwasawa factorization is the same for both of these. So we can
replace $\hat \Psi$ by $\hat \Psi^\prime$ and we have
\bdm
(\hat \Psi^\prime)^{-1} \dd \hat \Psi^\prime = \bbar 0 & a \\  \beta^{-2}B & 0 \ebar \lambda^{-1} \dd z + o(\lambda),
\edm
where $a = \frac{H}{4}(-s + i t)$ on $J$. The new singular frame
$\hat G_\omega^\prime$, which is obtained from the factorization of
$\hat \Psi^\prime \omega_1^{-1} =: \hat \Psi^\prime_\omega = \hat G_\omega^\prime \hat B_\omega^\prime$ satisfies $\sym_1(\hat G_\omega^\prime) = \sym_1(\hat G_\omega) = \tilde f + \textup{translation}$,
 but the frame
now also satisfies, along $J$, the analogue of 
equations (\ref{tildeeqns}), replacing $A$ with $a=\frac{H}{4}(-s+it)$. But the frame
$\fomega$ constructed above for $f$ also satisfies the same equations.
This implies that 
\bdm
\fomega ^{-1} \, \hat G^\prime_\omega \big|_J = \hat T,
\edm
where $\hat T: J \to \uu_1$ commutes with the matrix {\small $\bbar i & -i \lambda  \\ i \lambda^{-1}  & -i \ebar$}.  A computation (using that all matrices are normalized
to $I$ at $z=0$), shows that $\hat T$ must be of the form
\bdm
\hat T = \bbar 1  -iR & i R \\ -iR & 1+iR \ebar, 
    \hspace{1cm} R: J \times \SSS^1 \to \real,
\edm
where $R$ depends  on the loop parameter $\lambda$. Now 
\beqas
\sym_1(\hat G^\prime_\omega) &=&  \sym_1 (\fomega \hat T)  \\
&=& \left. \frac{-1}{2H}\left(\fomega \hat T e_3 \hat T ^{-1} \fomega^{-1}
    + 2i \lambda (\frac{\partial}{\partial \lambda}\fomega) \fomega^{-1}
     + 2i \lambda \fomega (\frac{\partial}{\partial \lambda} \hat T) \hat T^{-1} \fomega^{-1} \right) \right|_{\lambda=1} \\
     &=& \sym_1(\fomega) - \left. \frac{1}{H}\left(\fomega
        \bbar  i R^2 &  R -  i R^2 \\
          R +  i R^2 & -  i R^2 \ebar \fomega^{-1}
         +  i \lambda \fomega \frac{\partial R}{\partial \lambda} 
           \bbar - i & i \\ - i & i \ebar \fomega^{-1} \right)\right|_{\lambda=1}.
 \eeqas
We can use the assumption that $\tilde f = f$ along $J$, that is,
 $\sym_1(\hat G^\prime_\omega) = \sym_1(\fomega) + \textup{translation}$,
along  $J$.  Since all maps are normalized to the identity at $z=0$, 
this translation is actually the zero vector. It follows from this and the
formula for $\sym_1(\hat G^\prime_\omega)$ that
\bdm
 \left.\left( \bbar  i R^2 &  R -  i R^2 \\
          R +  i R^2 & -  i R^2 \ebar 
         +  i \lambda  \frac{\partial R}{\partial \lambda} 
           \bbar - i & i \\ - i & i \ebar  \right) \right|_{\lambda=1}
           = 0.
\edm
This gives the pair of equations
\beqas
\left. \left( i R^2 + \frac{\partial R}{\partial \lambda}
     \right) \right|_{\lambda=1}  = 0, \hspace{1cm}
\left. \left( R - i R^2 - 
 \frac{\partial R}{\partial \lambda} \right) \right|_{\lambda=1} = 0.
\eeqas
Hence $R \big|_{\lambda=1} = 0$, that is,
\bdm
 G^\prime_\omega \big|_J = F_\omega \big|_J = F_0.
\edm
But we already saw, in the paragraphs preceding this theorem, that,
given that we know the value of $a$ along $J$, 
the singular frame $\fomega$ is then uniquely determined by its value $F_0$ along
$J$.  Hence $\hat G^\prime_\omega = \fomega$, and $\tilde f = f$.
\end{proof}

\subsection{Example}
Choose $I=\real$, and the singular curve to be the helix in $\LL^3$ given by
$f_0(x) = [\sin(x), \, -\cos(x), \,  x]^T$, 
$f_x = [\cos(x), \, \sin(x), \, 1]$ and $v(x)=f_x(x)$.
Then $\theta(x)= x$, $s=t=1$ along $\real$.
We have $a=\frac{H}{4}(-1+i)$, $b=\frac{1}{2}iH$ and $r=\frac{1}{2}(1+H)$.
The singular potential is
\beqas
\xiom &=& \frac{H}{4} \left\{ \bbar (1-i) \lambda^{-2} & 
            (-1+i)\lambda^{-1} \\
              (1-i) \lambda^{-3} +  2 i \lambda^{-1} & 
              -(1-i) \lambda^{-2} ) \ebar
        + \bbar 2i\left(1+\frac{1}{H}\right) & 0 \\ 0 & -2i \left(1+\frac{1}{H}\right) \ebar  \right. \\
 &&       + \left. \bbar -(1+i) \lambda^2 & -2 i \lambda + (1+i) \lambda^3 \\
              -(1+i) \lambda & (1+i) \lambda^2 \ebar \right\} \dd z.
\eeqas
The corresponding translated frame, $\phihat = \phiom \omega_1$ has, from 
equation (\ref{phiinvdphi}), standard potential:
\bdm
\phihat^{-1} \dd \phihat = \frac{H}{4} \bbar \frac{2 i}{H}  & (-1+i)\lambda^{-1} - 2i \lambda
    + (1+i) \lambda^3 \\    
      -4i \left(1+\frac{1}{H}\right) \lambda^{-1} & -\frac{2i}{H}  \ebar \dd z.
\edm


\section{Identifying singularity types via the Bj\"orling construction}
In this section we find the conditions on the Bj\"orling data for the surface
constructed to have a cuspidal edge, swallowtail or cuspidal cross cap singularity
in a neighbourhood of a singular point. 
 If one considers non-degenerate 
$H$-surfaces parameterized by germs of their Bj\"orling data at some point,
then one can see that these are the generic singularities within this class.
 However, see the comments in Section \ref{openquestions}.

We first show that every weakly non-degenerate
$H$-surface is a frontal, and then use the criteria in \cite{krsuy} and \cite{fsuy}
for a frontal to have these types of singularities.
Examples are illustrated in Figure \ref{figure2}.

\subsection{The Euclidean normal to a generalized $H$-surface}  \label{normalsection} 
The commutators of our basis matrices satisfy $[e_1,e_2] = -2 e_3$,
$[e_2,e_3] = 2 e_1$, and  $[e_3,e_1] = 2 e_2$, and from this it follows that
the Euclidean cross-product 
on the vector space $\real^3$ corresponding to $\LL^3$ is given by
\bdm
A \times B = -\frac{1}{2} \Ad_{e_3} [ A, B],
\edm
where $\left[\, , \,\right]$ is the matrix commutator, and $Ad_X$ denotes
conjugation by $X$. Let $\left \| \, \cdot \, \right \|_E$ denote the 
standard Euclidean norm on $\real^3$. 

Let $f$ be a generalized 
$H$-surface with holomorphic frame $\phihat$. Since $f_x$ and $f_y$ are
parallel at singular points, the cross-product of these vanishes
there. 
Recall that the big cell is the union of two disjoint open sets,
$\bigcell = \bigcell^+ \cup \bigcell^-$.
 It turns out that one achieves continuity across the
singular set $C$ by defining, on $\Sigma^\circ$,
 the \emph{Euclidean (unit) normal} as follows:
\bdm
\enormal \,(z) := 
         \varepsilon  \frac{f_x \, \,  \times \,  f_y}{\left\| f_x \, \, \times \,  f_y \right\|_E} \, (z), \hspace{1cm} \varepsilon(z) = \pm 1, \, 
             \textup{ for } z \, \in \, \phihat^{-1}(\bigcell^\pm).
\edm
The two sets $\phihat^{-1}(\bigcell^\pm)$ are open and disjoint, so $\enormal$ is
a real analytic vector field on $\Sigma^\circ$.

\begin{lemma} \label{normallemma}
Let $f: \Sigma_s \to \LL^3$ be a weakly non-degenerate generalized $H$-surface.
Then the Euclidean unit normal extends across $C = \phihat^{-1}(\mathcal{P}_1)$ to give a real analytic vector
field on $\Sigma_s$.  At a point $z_0 \in C$, if coordinates are chosen so that
the singular holomorphic frame
$\phiom$ defined in Theorem \ref{thm2} satisfies $\phiom(z_0) = I$,
then the Euclidean normal is given at $z_0$  by
\beq \label {eucnorm}
\enormal (z_0) = \frac{1}{\sqrt{2}} (e_2 + e_3).
\eeq
If $\fomega$ is the singular frame obtained from $\phiom$ then, at nearby singular values  $z\in C$, the Euclidean normal is the unit vector in the direction of
\beq \label{ntilde}
\widetilde \enormal =  \Ad_{e_3} F_\omega (-e_2 + e_3) F_\omega^{-1}.
\eeq
\end{lemma}

\begin{proof}
On a neighbourhood, $\Omega \subset \Sigma$, of $z_0 \in C$ we can 
assume  by Theorem \ref{thm2} that $f$ is 
defined by a standard singular holomorphic frame $\phiom$ with $\phiom(z_0)=I$,
 with coordinates such that $z_0 = 0$, and  that there is an interval
$J = \Omega \cap \real$ containing $0$ such that $J \subset C$.
 On an open dense subset, $\Omega^\circ = \Omega \cap \Sigma^\circ$,
  of $\Omega$, we can Iwasawa factorize the
standard holomorphic frame $\phihat = \phiom \, \omega_1$ as
 $\phihat = \hat F \hat B$,
with $\hat F \in \uu$, $\hat B \in \uhat$. Now we have, 
\beqas
f_x &=& |f_x| \, F_C e_1 F_C^{-1}, \hspace{1cm} 
   f_y = |f_y| \, F_C e_2 F_C^{-1}\\
 &=& |f_x| \, F D e_1 D^{-1} F^{-1}, \hspace{1cm} 
   f_y = |f_y| \, F D e_2 D^{-1}F^{-1}
\eeqas
where $F_C$ and $D$ are given at equation (\ref{coordframe}),
and so $\enormal$ points in the direction of 
\beqas
\widetilde X &=& \varepsilon \left( F D e_1 D^{-1}F^{-1} \right)  \times   \left( FD e_2 D^{-1}F^{-1} \right)\\
&=&   \varepsilon \Ad_{e_3} (F\,  e_3 \, F^{-1}).
\eeqas
As in the proof of Lemma \ref{metriclemma},
  by Lemma \ref{switchlemma}, we have
 \beqas
\phihat &= &\hat F \hat B = \phiom \omega_1 = \fomega \hat B_\omega \omega_1\\
 &=& \fomega \hat K \bhat ^\prime,
\eeqas
where $\fhat = \varepsilon \fomega \hat K$, $~$
$\hat B_\omega = \bbar \rho & 0 \\ 0 & \rho^{-1} \ebar + \bbar 0 & \mu   \\ \nu & 0 \ebar \lambda + o(\lambda^2)$, and
$\rho: \Omega \to \real_+$, and $\mu$ 
and $\nu$ are $\C$-valued.
We also have
\bdm
\hat B_\omega \big|_J = I, \hspace{1cm}  \fomega(0) = I.
\edm
On $\Omega^\circ$ we can write
\bdm
\widetilde X = \varepsilon \Ad_{e_3} \left(F_\omega K e_3 K^{-1} F_\omega^{-1}\right).
\edm
Since $F_\omega$ is real analytic on the whole of $\Omega$, we only need
to analyze  $\widetilde Y:= \varepsilon K e_3 K^{-1}$. According to Lemma \ref{switchlemma},
we can choose $\hat K$ as
\beqas
\hat K = \bbar u & v \lambda \\ \varepsilon \bar v \lambda^{-1} & \varepsilon \bar u \ebar,\\
v = \frac{1}{\sqrt{\varepsilon h}}, \hspace{1cm} u = \varepsilon(\mu + \rho)\rho \bar v,\\
h := \left|\mu + \rho \right|^2 \left| \rho \right|^2 -1, \hspace{1cm} \left|u\right|^2-\left|v\right|^2 = \varepsilon.
\eeqas
Then 
\bdm
\widetilde Y = \varepsilon \bbar i \varepsilon (u \bar u + v \bar v) & -2i u v \\
            2i \bar u \bar v & -i \varepsilon(u \bar u + v \bar v) \ebar,
\edm
and
\beqas
\left\| \widetilde Y \right\|_E ^2 &=& (|u|^2 + |v|^2)^2 + 4|u|^2|v|^2   \\
 &=& (\varepsilon + 2 |v|^2)^2 + 4(\varepsilon + |v|^2)|v|^2\\
 &=&  1 + \frac{8}{h}\left(1  + \frac{1}{h}\right). 
\eeqas
The unit vector in the direction of $\widetilde Y$ is
\beqa
Y &=&  \left(1+ 8 h^{-1}\left(1  + h^{-1} \right)\right)^{-\frac{1}{2}} \widetilde Y  \nonumber\\
 &=& i \bbar Z_1 + Z_2 & -(\mu + \rho) \rho Z_1  \\
    (\bar \mu + \rho) \rho Z_1 & - Z_1 - Z_2 \ebar, \label{yhat}
\eeqa
where 
\beqa
Z_1 := \varepsilon 2 h^{-1} (1+ 8 h^{-1}(1  + h^{-1})^{-\frac{1}{2}}, ~
& ~\lim_{h \to 0} Z_1 = \frac{1}{\sqrt{2}}, 
~& ~
\lim_{h \to 0} \frac{\partial Z_1}{\partial y} = - \frac{1}{2\sqrt{2}}h_y, \label{zeqn1}\\
Z_2 := \varepsilon (1+ 8 h^{-1}(1  + h^{-1})^{-\frac{1}{2}},
~ & ~ \lim_{h \to 0} Z_2 = 0, ~ & ~
\lim_{h \to 0} \frac{\partial Z_2}{\partial y} =  \frac{1}{2\sqrt{2}}h_y. 
 \label{zeqn2}
\eeqa
Thus $Y$ is a well-defined real analytic vector field which, for real values of $z$, that
is when $h= \mu=0$ and $\rho = 1$, has the value
\bdm
Y(x) = \frac{1}{\sqrt{2}} \, (-e_2 +e_3).
\edm
Substituting this for $\varepsilon K e_3 K^{-1}$ in the expression for 
$\widetilde X$ above, gives the
stated formulae
for $\enormal(z_0)$ and $\widetilde \enormal(x)$.
\end{proof}

\begin{lemma} \label{enormalderlemma}
 Let $f$ be a generalized $H$-surface constructed from the Bj\"orling
data in Theorem \ref{bjorling}.  At $z=0$, the derivative $\dd \enormal$ of the
Euclidean unit normal is given by
\beq \label{enormalder}
 \dd \enormal  = -\frac{\theta_x}{\sqrt{2}} e_1 \dd x
   - \frac{Ht}{2 \sqrt{2}} (-e_2 + e_3) \dd y.
\eeq
\end{lemma}

\begin{proof}
We showed in  the previous lemma that $\enormal = \beta X$, for some real-valued
function $\beta$ and $X = \Ad_{e_3} (F_\omega Y F_\omega^{-1})$, where
$Y$ is given by equation (\ref{yhat}). We also have that $X(0) = \enormal(0)$, 
which means that $\beta(0) = 1$. 
Now $\left<\enormal, \dd \enormal \right>_E = 0$, and $X$ is parallel to $\enormal$,
 so it follows that
\beqas
\dd \enormal &=& \dd \beta X + \beta \dd  X \\
 &=&  \beta \left( \dd X - \left< \dd X , \enormal \right>_E \enormal \right),
\eeqas
 and we need to compute 
\bdm
\dd X = \Ad_{e_3}\left(F_\omega \left[F_\omega^{-1} \dd F_\omega \, , \, Y\right] F_\omega^{-1}\right)
+ \Ad_{e_3} \left( F_\omega  \,  \dd Y \, F_\omega^{-1}\right).
\edm
At $z=0$, we have, using $U_\omega$ and $V_\omega$ from (\ref{uhatandvhat}),
\bdm
F_\omega^{-1} (F_\omega)_x  = \frac{i \theta_x}{2} \bbar 1 & 0 \\ 0 & -1\ebar, \hspace{1cm}
F_\omega^{-1} (F_\omega)_y = i (U_\omega -  V_\omega) = 
   \bbar \frac{H s}{2}i & - \frac{Ht}{2} - \frac{Hs}{2}i \\
      - \frac{Ht}{2} + \frac{Hs}{2}i & - \frac{Hs}{2}i \ebar,
\edm
and, by the formulae $h_y=-4(r-\Im b)$, $ \mu_y = 2i \bar b$
and $\rho_y = -r$  from the proof of Theorem \ref{thm1},
\beqas
h= 0, &  h_x =0, & h_y =  -2 \theta_x,\\
\mu = 0, &  \mu_x=0, & \mu_y  =  Ht, \\
\rho=1, & \rho_x=0, & \rho_y =  -\frac{1}{2}(\theta_x + Ht).
\eeqas
Using these and the formulae (\ref{yhat})-(\ref{zeqn2}) one obtains, at $z=0$,
\bdm
X_x = -\frac{\theta_x}{\sqrt{2}} e_1,
  \hspace{1cm}
  X_y = -\frac{Ht}{\sqrt{2}} e_3.
\edm
Together with the value $\enormal = \frac{1}{\sqrt{2}} (e_2 + e_3)$ at $z=0$, and $\beta(0)=1$,
this gives the expression (\ref{enormalder}) for   $\beta \left( \dd X - \left< \dd X , \enormal \right>_E \enormal \right) \big|_{z=0}$.
\end{proof}


\begin{lemma}  \label{psiderivativelemma}
Let $f: \Sigma_s \to \LL^3$ be a generalized $H$-surface constructed by the
data in Theorem \ref{thm1}, with $\phiom(0) = \phihat(0) \omega_1^{-1} = I$. 
Set $s_0:= -\frac{4 \Re a(0)}{H}$ and $t_0 := \frac{4 \Im a(0)}{H}$ so that
\bdm
f_x = s_0 \, (-e_2 + e_3),
\hspace{1cm}
f_y = t_0  \, (-e_2 + e_3),
\edm 
Let $\psi: \Sigma_s \to \real$
be defined by 
\bdm
\psi = \varepsilon \|f_x \times f_y \|_E,
\edm
 where 
$\varepsilon(z)  = \pm1$, for $z \in \Phi^{-1}(\bigcell^\pm)$. Then at $z= 0$,
\beq
\dd \psi = \frac{16 |a|(\Im b - r)}{H} \sqrt{\frac{t_0^2 + s_0^2}{2}} \, \dd y.
\eeq
In particular, $\dd \psi(0) = 0 \Leftrightarrow \Im b(0) - r(0) = 0$.
\end{lemma}
\begin{proof}
At points away from  the real line, we have the decomposition 
$\phihat = \fhat \bhat$, and the coordinate frame found in Lemma
\ref{metriclemma} is:
$\fhat_C := \fhat  D$, with 
  $D = \textup{diag} \left(e^{i(\frac{\phi}{2} + \frac{\pi}{4})}, \,\, e^{-i(\frac{\phi}{2} + \frac{\pi}{4})} \right)$,
   and $a = |a|e^{i\phi}$.
The metric is given by $\dd s^2 = 4 g^2 \,\, (\dd x^2 + \dd y^2)$, with
$g  =  \varepsilon \frac{\chi^2 |a|}{H}$ and
$\chi = \sqrt{\left| |\mu + \rho|^2 - \rho^{-2}\right|}$.
And we have:
\beq  \label{fxfyn}
f_x = 2 \varepsilon g F_C e_1 F_C^{-1}, \hspace{1cm} 
f_y = 2 \varepsilon g F_C e_2 F_C^{-1}, \hspace{1cm} 
N = F_C e_3 F_C^{-1},
\eeq
where $N$ is the Lorentzian unit normal. Now 
\beqas
f_x \times f_y &=& -\frac{1}{2} \Ad_{e_3} [ f_x , f_y] \\
 &=& 4 g^2 \Ad_{e_3} F_C e_3 F_C^{-1},
\eeqas
so we can write $\psi = \varepsilon \|f_x \times f_y\|_E$ as
\bdm
\psi = 4 g \, \Gamma, \hspace{1cm} \Gamma:= \varepsilon g \, \| N \|_E.
\edm
Although $g \to 0$ and $\| N\|_E \to \infty$ as $z \to \real$, we can 
get an explicit expression for the product $\Gamma$. 
Writing 
$F_C = {\tiny \bbar A & B \\ \varepsilon \bar B & \varepsilon \bar A \ebar}$, the 
equations (\ref{fxfyn}) then imply that, as $z \to 0$, we have the finite limits:
\beqas
g \, \Im(A\bar B) \to -\frac{s_0}{4}, \hspace{1cm}
\varepsilon g\,(A^2-B^2) \to -i\frac{s_0}{2},\\
 g \, \Re(A \bar B) \to -\frac{t_0}{4}, \hspace{1cm}
\varepsilon g\,(A^2+B^2) \to -\frac{t_0}{2},
\eeqas
which imply
\bdm
\varepsilon g  A^2 \to -\frac{1}{4}(t_0 + i s_0), \hspace{1cm}
\varepsilon g B^2 \to \frac{1}{4}(-t_0 + i s_0).
\edm
Now 
\bdm
N = i  \bbar \varepsilon(|A|^2 + |B|^2) & - 2 A B \\
     2 \bar A \bar B & - \varepsilon(|A|^2 + |B|^2 \ebar,
\edm
so 
\beqas 
\Gamma &=&  \varepsilon g \|N\|_E =  \varepsilon g \left( (|A|^2 + |B|^2)^2 + 4|A|^2|B|^2 \right)^{\frac{1}{2}}  \\
&=& \left( g^2(|A|^4 + |B|^4 + 6|A|^2|B|^2 \right)^{\frac{1}{2}},
\eeqas
 \bdm
\lim_{z\to 0} \Gamma 
  =\sqrt{\frac{t_0^2 + s_0^2}{2}}.
 \edm
This limit is non-zero because $a$ is non-vanishing.

Similarly, the terms $\varepsilon g A^2$ and $\varepsilon g B^2$ 
also have well defined derivatives as $z \to \real$, following from the second derivatives
of $f$.  Since  $\varepsilon g A^2$ and $\varepsilon g B^2$ are non-zero at $z=0$, their absolute values are also differentiable there.
 Hence the derivative
$\dd \Gamma$ has a well defined finite limit as $z \to 0 \in \real$.

 Returning to $\psi = 4 g \Gamma$,
we have
\bdm
\dd \psi(0) = \lim_{z\to 0} \left( 4 \dd g \, \Gamma + 4g \,\dd \Gamma \right).
\edm
 Lemma \ref{metriclemma} informs us that 
$\lim_{z \to 0} g = 0$ and 
$\lim_{z \to 0} 
\frac{\partial g}{\partial y} = \frac{4 |a|(\Im b - r)}{H}$,
from which the claim of the lemma follow.
\end{proof}

\subsection{Frontals and fronts}

Let $U$ be a domain of $\real^2$.  A map $f: U \to \EEE^3$, into the three-dimensional
Euclidean space, is called a \emph{frontal} if there exists a unit vector field $\enormal: U \to \SSS^2$, such that $\enormal$ is perpendicular to $f_*(TU)$ in $\EEE^3$.   The map
$L = (f, \, \enormal): U \to \EEE^3 \times \SSS^2$ is called a 
Legendrian lift of $f$. If $L$ is an immersion, then $f$ is called a \emph{front}.  
A point $p \in U$ where a frontal $f$ is not an immersion is called 
a \emph{singular point} of $f$.

Suppose that the restriction of a frontal $f$, to some open dense set,
  is an immersion, and for some given  Legendrian lift $L$ of  
$f$, there exists a smooth function $\psi:U \to \real$
 such that, in local coordinates $(x,y)$,
\bdm
f_x \times f_y = \psi \enormal.
\edm
 Then a singular point $p$ is called \emph{non-degenerate}
if $\dd \psi$ does not vanish there.   In this situation, the frontal $f$ is called non-degenerate if every singular point is non-degenerate.

\begin{lemma}
Let $f: \Sigma_s \to \LL^3$ be a weakly non-degenerate generalized $H$-surface.
Let $\enormal$ denote the Euclidean unit normal defined in Section \ref{normalsection}.
Let $\EEE^3$ denote the vector space $\LL^3$ with the standard Euclidean inner product
$\left< \, \cdot \, \right>_E$. Then the map $f: \Sigma_s \to \EEE^3$, together with
the Legendrian lift $L = (f, \enormal): \Sigma \to \EEE^3 \times \SSS^2$,
defines a frontal.  The surface is non-degenerate as an $H$-surface, in accordance with
Definition \ref{degendef}, if and only if it is non-degenerate as a frontal.
\end{lemma}
\begin{proof}
By Lemma \ref{normallemma}, the map $\enormal: \Sigma_s \to \SSS^2$ is well defined and
real analytic, and so  $L = (f,\enormal)$ is a real analytic Legendrian lift
of $f$; in particular, $f$ is a frontal.   Regarding degenerate points, the map
$\psi$  above is the signed Euclidean norm $\varepsilon \| f_x \times f_y \|_E$, 
discussed in Lemma \ref{psiderivativelemma}, and we showed there that $\dd \psi$
vanishes at a singular point if and only $\Im b-r$ does. The latter expression is,
according to Theorem \ref{thm1}, the derivative of the function $h$, which was 
used previously to define degeneracy.
\end{proof}


\begin{lemma} \label{frontcondlemma}
Let $f$ be a non-degenerate generalized $H$-surface constructed from the
Bj\"orling data in Theorem \ref{bjorling}. Then $f$ is a front on a neighbourhood
of $z=0$ if and only if 
\bdm
t(0) \neq 0.
\edm

\end{lemma}
\begin{proof}
According the assumptions of the 
Bj\"orling construction, $\dd f = s(0)(-e_2 + e_3) \dd x + t(0) (-e_2 + e_3) \dd y$.
By Lemma \ref{enormalderlemma}, $\dd \enormal = \frac{-\theta_x}{2} e_1 \dd x + \frac{Ht(0)}{2\sqrt{2}}(e_2-e_3) \dd y$. It follows that the map $\dd L = (\dd f, \dd \enormal)$ 
has rank $2$ at $0$ if and only if $t(0) \neq 0$.
\end{proof}

\subsection{Cuspidal edges and swallowtails}

At a non-degenerate singular point, there is a well-defined
direction, that is a non-zero vector $\eta \in T_pU$, unique up to scale,
 such that $\dd f(\eta) = 0$, called the \emph{null direction}.  

A test for whether a singularity on a front is a swallowtail or a cuspidal edge is
given in \cite{krsuy}:
\begin{proposition} \label{swallowtailprop}
(\cite{krsuy}). 
Let $f: U \to \real^3$ be a front, and $p$ a non-degenerate singular point.
Suppose that $\gamma: (-\delta, \delta) \to U$ is a local parameterisation 
of the singular curve, with parameter $x$ and tangent vector $\gamma^\prime$, 
and  $\gamma(0)= p$,.
 Then:
\begin{enumerate}
\item The image if $f$ in a neighbourhood of $p$ is diffeomorphic to a
cuspidal edge if and only if $\eta(0)$ is not proportional to $\gamma^\prime(0)$.
\item The image if $f$ in a neighbourhood of $p$ is diffeomorphic to a
swallowtail if and only if $\eta(0)$ is proportional to $\gamma^\prime(0)$ and
\bdm
\frac{\dd}{\dd x} \det \left(\gamma^\prime(x), \eta(x)\right)\Big|_{x=0} \neq 0.
\edm
\end{enumerate}
\end{proposition}
We can use this test to prove the following result:

\begin{theorem}\label{swallowtailthm}
Let $f$ be a non-degenerate generalized $H$-surface constructed from the
Bj\"orling data in Theorem \ref{bjorling}. Then:
\begin{enumerate}
\item
$f$ is locally diffeomorphic to a cuspidal edge at $z_0=0$ if and only if 
\beqas
t(0) \neq 0 & \textup{and} & s(0) \neq 0.
\eeqas
\item
$f$ is locally diffeomorphic to a swallowtail at $z_0=0$ if and only if 
\bdm
\begin{array}{cccc}
t(0) \neq 0, &  s(0) =0 & \textup{and} & \frac{\dd}{\dd x}s(0) \neq 0.
\end{array}
\edm
\end{enumerate}
\end{theorem}
\begin{proof}
By Lemma \ref{frontcondlemma}, $f$ is a front at $z=0$ if and only if $t(0) \neq0$, so we can use the proposition above. We also have, along $J$,
\bdm
 f_x  = s  F_0 \, (-e_2 + e_3) \, F_0^{-1}, \hspace{1cm} 
   f_y = t F_0 \, (-e_2 + e_3) \, F_0^{-1}, 
\edm
and the null direction is
\beq  \label{nulldirection}
\eta (x) = t(x) \frac{\partial}{\partial x} - s(x) \frac{\partial}{\partial y}.
\eeq
Writing $x+iy = [x,y]^T$, the singular curve is given by $\gamma(x) = [x,0]^T$ and 
the null direction by $\eta(x) = [t(x),-s(x)]^T$, and so 
 the criteria in Proposition \ref{swallowtailprop} imply the claim.
\end{proof}

\subsection{Cuspidal cross caps}

 From \cite{fsuy} (Theorem 1.4), one has the following test for whether a non-degenerate frontal
 is locally  a cuspidal cross cap:
 
\begin{theorem} (\cite{fsuy}.) \label{fsuythm}
Let $f: U \to \real^3$ be a  frontal, with Legendrian lift $L = (f, \enormal)$,
 and let $z_0$ be a non-degenerate singular point.
Let $X: V \to \real^3$ be an arbitrary differentiable function on a neighbourhood $V$ of $z_0$
such that:
\begin{enumerate}
\item  \label{transcond1}
$X$ is orthogonal  to  $\enormal$.
\item \label{transcond2}
$X(z_0)$ is transverse to the subspace $f_*(T_{z_0}(V))$.
\end{enumerate}
Let  $x$ be the parameter for the singular curve, and set
\bdm
\tilde \psi (x) := \left< \enormal, \, \dd X  (\eta) \right>_E \big|_x.
\edm
The frontal $f$ has a \emph{cuspidal
cross cap} singularity at  $z=z_0$ if and only:
\begin{itemize}
\item [(A)] \label{conda} $\eta(z_0)$ is transverse to the singular curve;
\item [(B)] \label{condb} $\tilde \psi(z_0) = 0$ and $\tilde \psi^\prime(z_0) \neq 0$.
\end{itemize}
\end{theorem}

\begin{theorem} \label{crosscapsthm}
Let $f$ be a non-degenerate $H$-surface constructed from the
Bj\"orling data in Theorem \ref{bjorling}. Then 
$f$ is locally diffeomorphic to a cuspidal cross cap around $z=0$ if and only
if the following conditions hold:
\bdm
\begin{array} {cccc}
s  (0) \neq 0, &  t (0) = 0 & \textup{and} &
     \frac{\dd }{\dd x} t(0) \neq 0.
\end{array}
\edm
\end{theorem}
\begin{proof}
In a neighbourhood of
$0$, the singular curve is given by an interval $J =(x_1,x_2)$ of the real line. 
Recall from the proof of Lemma \ref{enormalderlemma} that we found the following
formula for $\enormal$: 
\beqas
\enormal = \beta \Ad_{e_3}(F_\omega Y F_\omega^{-1}) && Y = a e_1+  b e_2 + c e_3, \\
a = \Im (\mu) \rho Z_1,  & b = - (\Re (\mu) + \rho) \rho Z_1, 
   & c = Z_1 + Z_2,
\eeqas
and along $J$ we have:
$Z_1 = \frac{1}{\sqrt{2}}$, $Z_2 = 0$,  
$a=0$, $b = -1/\sqrt{2}$ and $c = 1/\sqrt{2}$
so that $Y = \frac{1}{\sqrt{2}} (-e_2+e_3)$ for real values
of $z$.

 We will apply Theorem \ref{fsuythm} with 
 the vector field defined by the cross product:
\beqas
X &=& 
\left( \Ad_{e_3} F_\omega \, e_2 \, F_\omega^{-1} \right) \, \times \, \left(  \Ad_{e_3}F_\omega Y F_\omega^{-1} \right)   \nonumber \\
  &=&  -\frac{1}{2} 
 F_\omega  [e_2  \, , \, a e_1 + b e_2 + c e_3 ]F_\omega^{-1} \nonumber \\
 &=& - F_\omega \left( c e_1 + a e_3   \right) F_\omega^{-1}.  
 \eeqas
 $X$ is orthogonal to $\enormal$ because $\Ad_{e_3}F_\omega Y F_\omega^{-1}$ is proportional
to $\enormal$.
Along $J$  we have
\beqas
 f_x  = s  F_0 (-e_2+e_3) F_0^{-1}, \hspace{1cm} 
   f_y = t F_0 (-e_2+e_3) F_0^{-1}, \hspace{1cm}
 X =  
 -\frac{1}{\sqrt{2}} F_0 e_1 F_0^{-1}, 
\eeqas
so $X$ is transverse to $f_*(T_{z_0}(V))$.
That is, $X$ satisfies conditions \ref{transcond1} and \ref{transcond2}
of Theorem \ref{fsuythm}.

Now consider the conditions (A) and (B).
The null direction along $J$ is given by
$\eta = t \frac{\partial}{\partial x} - s \frac{\partial}{\partial y}$,
and this is transverse to the singular curve at $z_0=0$ if and only if 
$s(0) \neq 0$, so our first condition is equivalent to condition (A).

 To investigate $\tilde \psi$, we need an expression for  $\left< \enormal \, , \, \dd X \right>_E$ along $J$.
Now
\bdm
\dd X = -c \, \dd (F_\omega \, e_1 \, F_\omega^{-1}) -
  a  \, \dd (F_\omega \, e_3 \, F_\omega^{-1} ) - 
      \dd c \,  \Ad_{F_0} e_1 - \dd a \,  \Ad_{F_0} e_3.
\edm
Along $J$ we have $\dd a = \dd (\Im \mu) \cdot \frac{1}{\sqrt {2}} \cdot 1 = 0$,
because we earlier computed $\dd \mu = H t \dd y$ which is real.  We also have
$a =0$, and $\left< \enormal \, , \, \Ad_{F_0} e_1 \right> _E = 
\left< \frac{1}{\sqrt{2}}(e_2 + e_3) \,,\, e_1 \right> _E = 0$. We used that
$F_0$ takes values in $SU(2)$ and so preserves the Euclidean inner product.
Hence only the first term in the above expression for $\dd X$ contributes to
$\left< \enormal \, , \, \dd X \right>_E$:
\bdm
\left< \enormal \, , \, \dd X \right>_E \big |_J = 
- \frac{1}{\sqrt{2}}\left< \enormal \, , \, \dd (F_\omega e_1 F_\omega^{-1}) \right> \big |_J.
\edm
To compute this, we use: 
\bdm
F_\omega^{-1} (F_\omega)_x =  U_\omega +  V_\omega, \hspace{1cm} F_\omega^{-1} (F_\omega)_y = i(U_\omega -  V_\omega),
\edm
where, from equation (\ref{uhatandvhat}), at $\lambda = 1$, 
\bdm
U_\omega + V_\omega = \bbar -2i \Im a + ir & 2 i \Im a + \bar b \\
          -2i \Im a + b & 2 i \Im a -i r \ebar, \hspace{.3cm}
i(U_\omega - V_\omega) = \bbar - 2 i \Re a & 2 i \Re a - i \bar b \\
            - 2 i \Re a + i b & 2 i \Re a \ebar.
\edm
With this and  $s = \frac{-4 \Re a}{H}$, $t = \frac{4 \Im a}{H}$,  $b = \frac{1}{2}iHt$, and $r = \frac{1}{2}(\theta_x + Ht)$ one obtains along $J$
\beqas
(F_\omega e_1 F_\omega^{-1})_x &=&  F_0 \, [ \, U_\omega + V_\omega, e_1 \,]\, F_0^{-1} \\
&=& F_0 \bbar 4i \Im a -2i \Im b & -4 i \Im a + 2 ir \\
   4i \Im a - 2 i r & -4i \Im a + 2i \Im b \ebar F_0^{-1} \\
   &=& F_0 \,\theta_x e_2 \, F_0^{-1},
\eeqas
\beqas
(F_\omega e_1 F_\omega^{-1})_y &=& F_0 \,[\,i(U_\omega-V_\omega), e_1 \,]\, F_0^{-1}\\
&=& F_0 \bbar 4 i \Re a - 2 i \Re b & - 4 i \Re a \\
        4 i \Re a & -4 i \Re a + 2 i \Re b \ebar F_0^{-1}\\
   &=& F_0 \, Hs(e_2 -e_3) \, F_0^{-1}.
\eeqas
Hence we obtain the following expression along $J$,
\beqas  
\tilde \psi &=& \left<\enormal, \dd X(\eta) \right>_E \big|_x \\
&=& - \frac{1}{\sqrt{2}} \frac{1}{\sqrt{2}} \left< (\Ad_{e_3} \Ad_{F_0}(-e_2 + e_3) \,,\, tX_x - s X_y \right>_E \\
&=& -\frac{1}{{2}} \left< e_2 + e_3 \,,\, t \theta_x e_2 -s^2H(e_2 -e_3) \right>_E\\
 &=&  -\frac{1}{2} t \, \theta_x.
\eeqas

 Condition (B) of Theorem \ref{fsuythm} is thus equivalent to the pair
of equations
\beqas
t \, \theta_x \Big|_{x=0} = 0, && 
\left[\frac{\dd t }{\dd x} \, \theta_x + t \, \frac{\dd \theta_x}{\dd x}\right]_{x=0} \neq 0 .
\eeqas
Since $\theta_x \neq 0$, this pair of equations is equivalent to 
$t(0) = 0$ and $\frac{\dd t}{\dd x}(0) \neq 0$.
\end{proof}


\providecommand{\bysame}{\leavevmode\hbox to3em{\hrulefill}\thinspace}
\providecommand{\MR}{\relax\ifhmode\unskip\space\fi MR }
\providecommand{\MRhref}[2]{%
  \href{http://www.ams.org/mathscinet-getitem?mr=#1}{#2}
}
\providecommand{\href}[2]{#2}

\end{document}